\documentclass[a4paper]{article}
\usepackage[margin=25mm]{geometry}
\usepackage{amsmath}
\usepackage{amsthm}
\usepackage{amsfonts}
\usepackage{amssymb}
\usepackage{graphicx}
\usepackage{array}
\usepackage{bm}
\usepackage{wrapfig}
\usepackage{tabularx}
\usepackage{enumerate}% http://ctan.org/pkg/enumerate
\usepackage{verbatim}
\usepackage{authblk}
\newcommand{\be}{\begin{equation}}
\newcommand{\ee}{\end{equation}}
\newcommand{\bea}{\begin{eqnarray}}
\newcommand{\eea}{\end{eqnarray}}
\newcommand{\bean}{\begin{eqnarray*}}
\newcommand{\eean}{\end{eqnarray*}}
\newcommand{\ben}{\begin{equation}{nonumber}}
\newcommand{\een}{\end{equation}{nonumber}}

\newcommand{\case}[1]{\subsubsection*{Case #1}}
\newcommand{\procedure}[1]{\paragraph*{Procedure #1}}
\numberwithin{equation}{section}

\theoremstyle{definition}
\newtheorem{dfn}{Definition}[section]
\newtheorem{definition}[dfn]{Definition}
\newtheorem{xmpl}[dfn]{Example}
\newtheorem{example}[dfn]{Example}

\theoremstyle{theorem}
\newtheorem{theorem}[dfn]{Theorem}

\newtheorem{thm}[dfn]{Theorem}
\newtheorem{lema}[dfn]{Lemma}
\newtheorem{pro}[dfn]{Proposition}
\newtheorem{proposition}[dfn]{Proposition}
\newtheorem{coro}[dfn]{Corollary}
\newtheorem{corollary}[dfn]{Corollary}
\newcommand{\bdfn}{\begin{dfn}}
\newcommand{\edfn}{\end{dfn}}
\newcommand{\bthm}{\begin{thm}}
\newcommand{\ethm}{\end{thm}}
\newcommand{\blema}{\begin{lema}}
\newcommand{\elema}{\end{lema}}
\newcommand{\bpro}{\begin{pro}}
\newcommand{\bcoro}{\begin{coro}}
\newcommand{\bxmpl}{\begin{xmpl}}
\newcommand{\brmrk}{\begin{rmrk}}

\theoremstyle{remark}
\newtheorem{rmrk}[dfn]{Remark}
\newtheorem{remark}[dfn]{Remark}

\title{Properties of glued knots}
\author[1]{Shane D'Mello}
\author[1]{Vinay Gaba}
\affil[1]{Indian Institute of Science Education and Research, Mohali, India}

\begin{document}

\maketitle

\begin{abstract}
  We study the properties of, glued knots, a sub-class of real rational knots that can be constructed by gluing ellipses. We define an invariant called the gluing degree and relate it to various classical properties of knots and classify all knots up to gluing degree~6.
\end{abstract}

\section{Introduction}
The study of real algebraic knots is one of the tangential questions motivated by Hilbert's Sixteenth problem. Until now, little is known about the classical knot theoretic properties of such knots and the question seems hard to answer in general.

In \cite{bjorklund}, Bj\"orklund introduced the method of constructing a real rational knot of degree~$d_1+d_2$ by gluing two real rational knots of degrees $d_1$ and $d_2$.  We recall the gluing construction:

\begin{thm}[Bj\"orklund]
  Consider two oriented real rational knots $k_1$ and $k_2$ of degree $d_1$ and $d_2$ respectively in $\mathbb{RP}^{3}$ which intersect transversally in a point $p$. Therefore, in a small enough neighbourhood of $p$ the union is isotopic to a pair of intersecting straight line segments. Consider the knot obtained by smoothing the pair of segments in a manner compatible with the orientation, then the resulting knot has a real rational representative of degree~$d_1 + d_2$ which lies completely inside a small tubular neighbourhood of the union of $k_1$ and $k_2$.
\end{thm}

In \cite{ramashane}, it was shown that all classical knots have an isotopy class representative that can be constructed using this method and one can know the degree of the resulting real rational knot that is sufficient. %please see this line...ahead of that is sufficient%
However, we do not yet know if for any fixed degree~$d$, one can construct \emph{all} real rational knots of degree~$d$ using this method. Nevertheless, this method has been used to construct all known real rational knots of low degrees and since the subset of real algebraic knots of degree~$d$ that can be constructed by this method are geometric, there should be stronger connections with its topological properties. So it is worth studying this subset on its own.

Since in this article we are interested in classical knot theoretic properties, in what follows, we restrict our attention to affine real rational knots, which are defined as real rational knots that avoid at least one plane in $\mathbb{RP}^3$ and, therefore, lie in a copy of $\mathbb{R}^3$ embedded in $\mathbb{RP}^3$. All affine real rational knots are of even degree because, by Bezout's theorem, their intersection number with a plane is 0 mod 2.

The simplest non-trivial \emph{affine} real rational knot, up to rigid isotopy, is the ellipse, which is of degree~2. Gluing $k$ ellipses will result in a real rational knot of degree~$2k$. In \cite{ramashane}, it was shown that any knot in $\mathbb{R}^3$ has a real rational representative that can be constructed by gluing ellipses. All this motivates the following natural definition:

\begin{definition}
An (affine) glued knot of degree~$2m$ is a real rational knot that can be constructed by gluing $m$ ellipses.
\end{definition}

Note that the simplest non-trivial real rational knot is the straight line, but that is not affine and lives in $\mathbb{RP}^{3}$. We use ellipses instead because in this article we are focusing only on glued knots in $\mathbb{R}^3$ since they form classical knots, and leave the projective ones for future work.

Furthermore, while the construction was originally devised to construct examples of real rational knots, one can study them without keeping in mind that they necessarily form real algebraic knots of a given degree and can simply treat them as knots that are constructed out of oriented ellipses that are arranged in $\mathbb{RP}^{3}$ a certain way.

Since any classical knot in $\mathbb{R}^3$ can be constructed by gluing ellipses, we focus on the least degree:

\begin{definition}
A (classical) knot in $\mathbb{R}^3$ will be said to have gluing degree~$2m$ if it can be constructed by gluing $m$ ellipses but cannot be constructed with fewer than $m$ ellipses.
\end{definition}

% Has already been defined later
%\begin{definition}
% A preglued knot of degree~$2m$ is a union of $m$ ellipses, $P_1, P_2, \ldots, P_m$  so that $P_i \cap P_j$ is a single point if $j = i+1$ (we call this single point a glued point); otherwise it is empty. Furthermore, the tangents of the ellipses $P_i$ and $P_{i+1}$ are linearly independent in the tangent space of $\mathbb{R}^3$ at the point of intersection. An oriented preglued knot is a preglued knot where each ellipse is oriented.
%\end{definition}

Like the stick number that is defined in terms of polygonal representatives of knots, and the minimal crossing number that is defined in terms of diagrammatic representatives of knots, the gluing degree is also a natural invariant. Like these invariants it is typically hard to compute them for any given knot. Simply producing a representative only provides an upper bound and one has to demonstrate that the knot cannot have a representative with a lower degree. For simpler knots this can be done by classifying glued knots of low degrees. We will do so in section~\ref{classification} and thereby prove that both the trefoil and figure eight have gluing degree~6.   %%%Add some knots of degree 8
   This article also studies some properties of the gluing degree including its relationship with other invariants.

Unlike polygonal representatives, very low degree representatives of an isotopy class of a even simple knots like the trefoil have very non-standard diagrams, making these objects more challenging than their polygonal counterparts. %%%Give an example using a figure

Since glued knots are also real rational knots, in section~\ref{encomplexed} we also investigate the relationship of the gluing degree with the only known real algebraic knot invariants: the encomplexed writhe number (from real rational knots), and the real rational degree and prove a sharp bound on the encomplexed writhe number realized by certain torus knots.  We do not know if the same bound works for general affine real rational knots (which are not necessarily obtained by gluing). In~\cite{mikhalkin}, it was proved that all maximal writhed \emph{projective} real rational knots are isotopic to the projective version of torus knots, however we show that at least for \emph{affine} and glued knots of maximal writhe, this uniqueness is not necessarily true.

It is natural to ask how the gluing degree relates with other invariants of knots. Here we investigate the relationship with these classical knot invariants: the 3-coloring invariant, and the crossing number, and therefore also get some a bound on the span of the Jones polynomial. In section~\ref{3colouring}, we prove the existence of prime glued knots of degree $2k+4$ which are colourable in $3^k$ ways. It can be easily shown that the number of crossings of a diagram of a glued knot of degree~$2n$ is bounded above by $(n-1)(2n-1)$. So the minimal crossing number of a knot is related to its gluing degree~$2n$ by being bounded above by $(n-1)(2n-1)$. However, in section~\ref{rigid}, we prove that the span of the Jones polynomial of an affine glued knot of degree~$2n$ is bounded above by $(2n-1)(n-1)-1$ for prime knots, which is one less than what can be derived from the crossing information alone. This also shows that the crossing number of alternating prime glued knots is bounded above by $(2n-1)(n-1)-1$ as span of the Jones polynomial is equal to the crossing number in that case.

Seen in a certain way, gluing reminds us of the skein relation. Indeed, in section~\ref{skeinrigid} exploit this to arrive at a relationships between the degree of the Jones (respectively Alexander) polynomial of glued knots glued out of $m$ ellipses with the Jones (respectively Alexander) polynomial of $m$ interlinked ellipses, which we call rigid links.

Since any glued knot of degree~$2m$ is also a real rational knot of degree~$2m$, the gluing  degree is always bounded below by the real rational degree. It is still not certain if the real rational degree can be strictly less than the gluing degree, which would be equivalent to finding a real rational knot of a particular degree which is impossible to construct by gluing. Once again, classifying glued knots is a step in that direction.
A large  proportion of the article works out the classification of such knots of low degree. The techniques illustrated in the proofs can be used for higher degrees too.

Since a glued knot is also a real rational knot, we can talk about its writhe number. Here we prove a bound on the maximum writhe number of a glued knot, that is surprisingly much smaller than the maximum possible writhe number of a projective real rational knot, including ones glued out of lines. We do not yet know if maximally writhed affine real rational knots (not necessarily glued ones) have larger bounds on their degree.

%\cite{mikhalkin} proved that all maximally writhed real rational knots of a given degree are isotopic to each other. However, we show that maximally writhed glued knots of a given degree need not be isotopic to each other.%

%It can be easily shown that the number of crossings of a diagram of a glued knot of degree~$2m$ is bounded above by $(d-1)(2m-1)$. So the minimal crossing number of a knot is related to its gluing degree~$2m$ by being bounded above by $(d-1)(2m-1)$. However, we will show that this bound can be reduced by 1.%

\subsection*{Acknowledgments} The first author (D'Mello) is grateful for the MATRICS grant MTR/2018/000728 from SERB (Govt. of India) for supporting this work. The second author (Gaba) is grateful for the IISER Mohali PhD fellowship for its support.

\section{Properties of knots}

\subsection{3-colouring invariant}
\label{3colouring}
A trefoil is 3-colourable in $3^2$ ways and can be glued out of 3 ellipses and, therefore, has degree~6. If two knots, which are situated in $\mathbb{R}^3$ so that they intersect at only one point but are unlinked, are glued, then the glued knot is the connected sum of the individual knots. Therefore, if one takes the connected sum of $k$ copies of the trefoil, it will result in a knot of degree~$6k$ which 3-colourable in $3^{k+1}$ ways. Therefore, we have easily shown that we can always find a knot that is 3-colourable in $3^{k+1}$ ways of degree~$6k$ (and above), however, except for $k=1$, the knots are never prime. We now prove a similar theorem for prime affine knots:

\begin{figure}[h]
\centering
\includegraphics[width=0.5\textwidth]{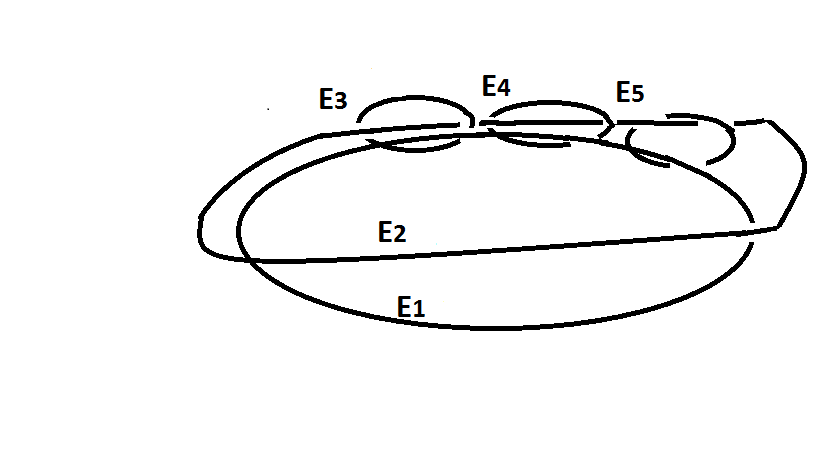}\hfill
\caption{gluing 5 ellipses to get $3^3$ number of 3-colors.}
\label{3_colors}
\end{figure}

\begin{figure}[h]
\centering
\includegraphics[width=0.5\textwidth]{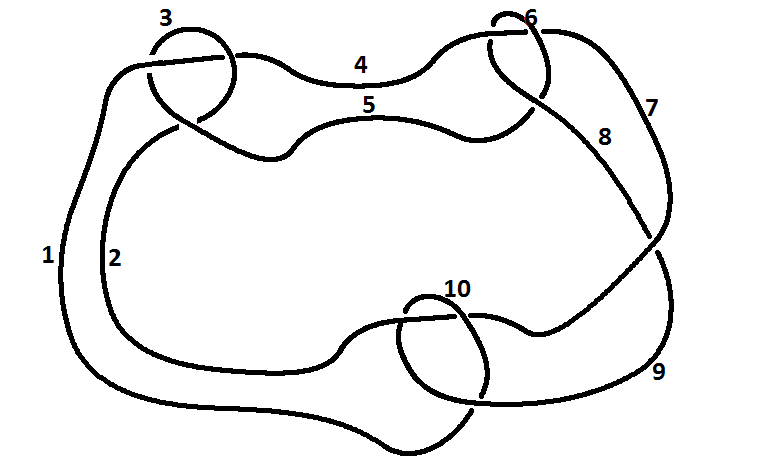}\hfill
\caption{prime alternating knot with $3^3$ number of 3-colours. } \label{murasugi_sum}
\end{figure}

\begin{thm}
  \label{3colourable}
There exists a prime knot of degree~$2k+4$ (and above) which is $3$~colourable in $3^k$ ways, for $k$$\geq3$.
\end{thm}

Before we begin the proof, we introduce a definition that will be used several times:
\begin{definition}
  \label{preglued}
  Consider a configuration of $m$ ellipses so that any two ellipses intersect in at most 1 point. We can associate a graph that encodes the intersections of these ellipses by representing each ellipse by a vertex and joining two vertices  by an edge if and only if the ellipses intersect.  We say this configuration forms a preglued knot of degree~$2m$ if and only if the associated graph is a tree.
  %$P_1, P_2, \ldots, P_m$  so that $P_i \cap P_j$ is a single point if $j = i+1$ (we call this single point a glued point); otherwise it is empty. Furthermore, the tangents of the ellipses $P_i$ and $P_{i+1}$ are linearly independent in the tangent space of $\mathbb{R}^3$ at the point of intersection. An oriented preglued knot is a preglued knot where each ellipse is oriented.
\end{definition}

Note that associated to an oriented preglued curve knot is a glued knot defined by smoothing out all the glued points in a manner compatible with the orientations. We now begin the proof of theorem~\ref{3colourable}.

\begin{proof}
  Consider the preglued knot shown in Figure~\ref{3_colors}. After smoothing It is isotopic to the knot shown in Figure~\ref{murasugi_sum}. If we fix the colour of arc 1 to be $c_1$, then we claim that we can choose any of the 3 colours to colour arc 3. If we denote its colour by $c_3$, then observe we have one of the two cases:
  \begin{enumerate}
  \item[Case 1] $c_3 = c_1$. Then arc 5 will also have to have the same colour because it meets arcs 1 and 3 at the same crossing. However, that forces arc 2 to also have the same colour because it shares a crossing with arcs 5 and 3. Therefore, arcs 2 and 1 share the same colour. Arcs 4 and 5 also share the same colour because they are each forced to share the colour with arc 3.
  \item[Case 2] $c_3 \neq c_1$. Then arc 5 will have to be coloured with the third colour different from $c_3$ and $c_1$. However, since arc 2 meets at the same crossing as arcs 5 and 3, it is forced to have colour $c_1$. Once again, arcs 2 and 1 share the same colour. Once again, arcs 4 and 5 share a colour because arc 4 meets arcs 3 and 1 at a crossing, but arc 5 meets arcs 3 and 2 at a crossing, and since arcs 1 and 2 share the same colour, arcs 4 and 5 cannot be coloured either $c_1$ or $c_3$.
  \end{enumerate}

  So we can choose any colour for arc 3, but no matter what colour we choose, arc 2 always has the same colour as arc 1. Furthermore, the choice of arc 3 fixes the colour for arcs 4 and 5. By a similar reasoning, arcs 4 and 5 share the same colour and one has a choice of 3 colours for arc 6. Continuing in this way, we could add $E_6$ and so on in line with $E_3$, $E_4$ and $E_5$ inductively in the same way and we get the result.

  We now prove that these knots are prime. We use the fact that an alternating knot is prime if it has an alternating prime diagram, i.e. any simple closed curve that cuts the diagram in only two points will enclose a simple arc.

  If a simple closed curve intersects the diagram in exactly 2 points, then the two points separate the curve into two arcs that have to each lie completely in a region that the diagram separates the plane into. If we assume that the knot is not prime then the simple closed curve has to enclose a part of the knot that is not a simple arc. Therefore, the two points cannot lie on the same ``edge'' of the diagram, where we consider the diagram as a 4-valent graph. Now to avoid intersection in any other points, the interior of each arc should be a subset of a region. Also, the regions have to be distinct otherwise the simple closed curve would not intersect the diagram transversally. But this implies that there are at least two regions which share two edges. That follows easily by inspection and induction.
\end{proof}

\subsection{Encomplexed writhe number}
\label{encomplexed}
Any glued knot of degree~$d$ is a real algebraic knot of degree~$d$ and any real algebraic knot has a natural rigid isotopy invariant defined as the encomplexed writhe number~\cite{viroencomplex}.  The maximal writhe number of a real rational knot can be easily proved to be $M_d:=\frac{(d-1)(d-2)}{2}$. Bj\"orklund showed that in $\mathbb{RP}^3$, the bound is sharp and in fact proved it by gluing lines. However, they were all projective knots. Here, we consider the affine counterpart but for the sub-case of glued knots: We prove a bound for the writhe number of affine glued knots and also prove that the bound is sharp. Interestingly, the bound is much smaller for affine glued knots.

\begin{thm}
  The maximum writhe number of an affine glued knot of degree~$2m$ is bounded above by $(m-1)^2$.
\end{thm}

\begin{proof}
  The maximal writhe is computed by adding up the local writhe at each crossing, which is either +1 or -1. However, crossings may also be non-real. In order to predict the local writhe at a non-real crossing, one can exploit the fact that the writhe number is a rigid isotopy and isotope the curve to one that can be projected to a diagram with only real crossings.

  Consider a knot glued out of ellipses. Every crossing from a projection of such a knot to a diagram, comes from some pair of ellipses. Consider the link formed by a pair of ellipses and fix a projection to examine the diagram of that link under the projection. Note that if we glue another real rational knot to either of those ellipses at points that do not project to crossing, it will not change the local writhes. %%explain
 Since we can always choose a projection so that the points of gluing are never crossings, we can compute the local writhes at a crossing by only considering the ellipses contributing to that crossing. Therefore, we have reduced the case to considering the possible local writhes of the crossings of two ellipses.

  We can always rigidly isotope a pair of ellipses to ensure that under the fixed projection they give rise to 4 real crossings. It is easy to check that the sum of the 4 local writhes is between -2 and 2.  If the two ellipses are glued to each other, resulting in one less local writhe, again by inspection we can see that the contribution of the writhe number is between -1 and 1.

  If a knot has been formed by gluing $m$ ellipses, then there are $m\choose 2$ pairs but out of these $m-1$ are glued pairs contributing to a maximum sum of local writhes of 1, while the others contribute to a maximum sum of local writhes of 2. Therefore, the total writhe number can be no more than $2{m\choose 2}-(m-1)$.
\end{proof}

The bound mentioned above on the encomplexed writhe number is sharp:

\begin{figure}[h]
\centering
\includegraphics[width=0.5\textwidth]{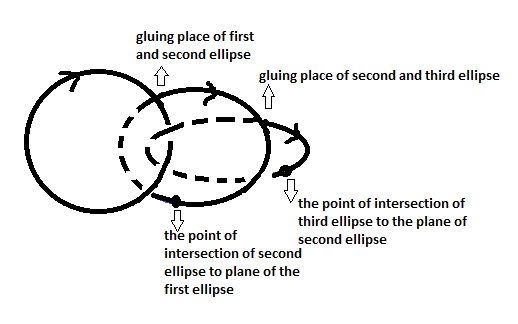}
  \caption{Preglued diagrams of ellipses with maximum writhe}
  \label{colored_diagram}
\end{figure}

\begin{figure}[h]
\centering
\includegraphics[width=0.5\textwidth]{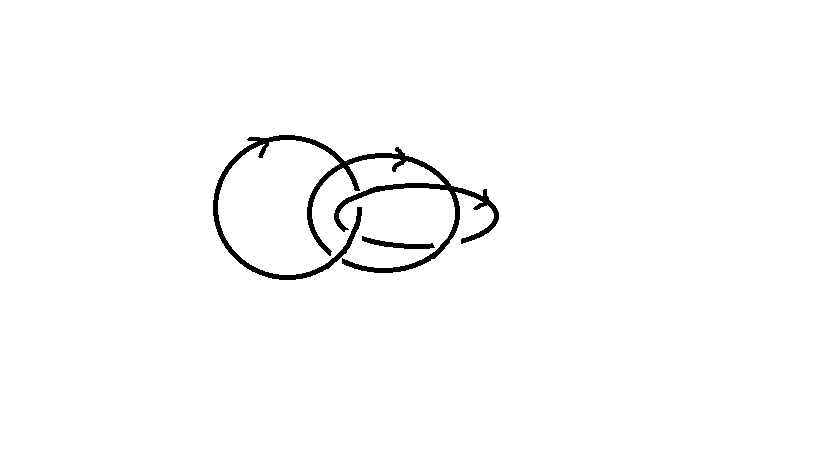}\hfill
\includegraphics[width=0.5\textwidth]{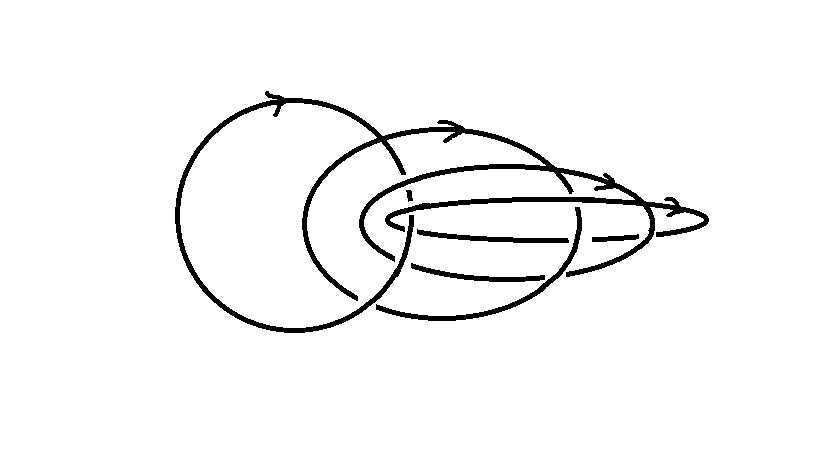}
  \caption{Preglued diagrams corresponding to glued knots with maximum writhe}
  \label{max_writhe}
\end{figure}

\begin{figure}[h]
\centering
\includegraphics[width=0.5\textwidth]{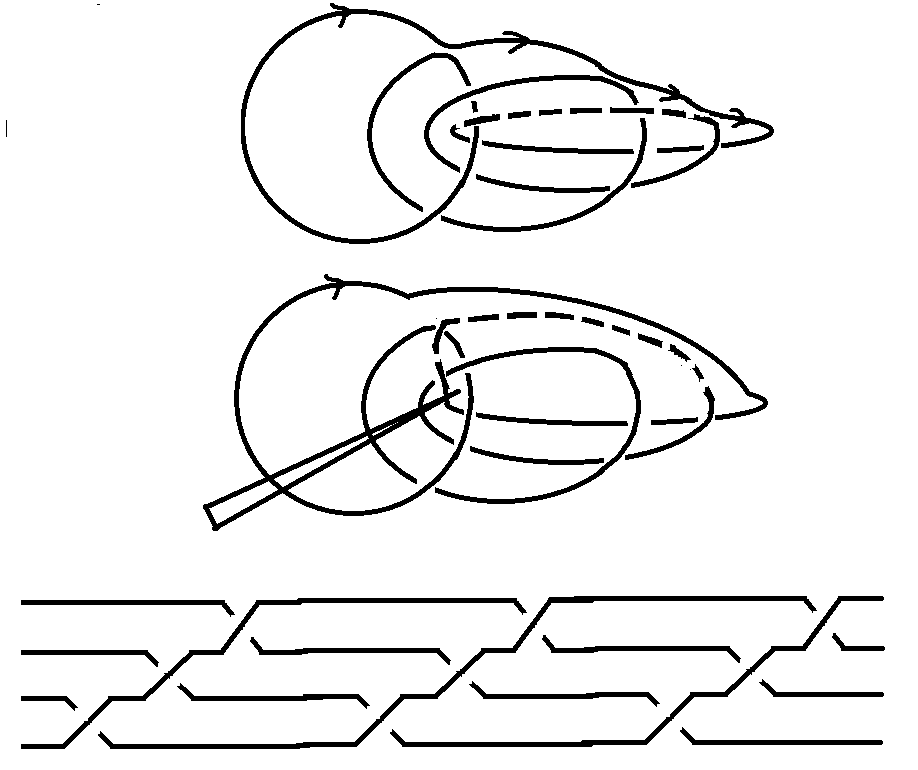}\hfill
\caption{If we perform some moves as shown in figure to a diagram with $m$ ellipses , we get the braid presentation of in the form $(\sigma _{1}\sigma _{2}\cdots \sigma _{{m-1}})^{m-1})$. And this is the braid presentation of the T(m,m-1).}
 \label{ReidemeisterMovesTorus}
\end{figure}

\begin{thm}
  For each $m=1, 2,\ldots$, there exists an affine glued knot of degree~$2m$ with the maximal possible writhe number, $(m-1)^2$, which is isotopic to the torus knot $T(m,m-1)$.
\end{thm}

\begin{proof}
  Figure~\ref{colored_diagram} and ~\ref{max_writhe} show how to construct family of knots, $K_m$, where $K_m$ is glued out of $m$ ellipses, by presenting the configuration before gluing, but the choice of gluing is indicated by the orientation. After performing the kind of moves shown in Figure~\ref{ReidemeisterMovesTorus}, it is straightforward to check that $K_m$ is  a $T(m,m-1)$ torus knot. We will prove, by induction, that the writhe number of $K_m$ is $(m-1)^2$, which is the maximum allowed writhe number of a knot glued out of $m$ ellipses.

  The main idea is to build $K_{m+1}$ by gluing an ellipse to $K_m$ in such a way that it links with all the $m-1$ ellipses that it is not glued to, and with the orientation chosen so as to make the local writhes of the new crossings positive. There are 2 crossings each with the ellipses that it is not glued to, but each crossing is of positive sign owing to linking. There is only one crossing with the ellipse that is glued to, also with positive sign. By the induction hypothesis, the writhe number of $K_m$ is $(m-1)^2$. The new ellipse adds $2(m-1) + 1$ to it, making the writhe number of $K_{m+1}$ equal to $(m+1 -1)^2$, as desired.
\end{proof}

\begin{rmrk}
The bound for affine glued knots is much smaller than the bound for projective ones. We do not yet know if the above bound on the encomplexed writhe number holds for general real rational knots too.  Indeed, \cite{mikhalkin} proved that maximal projective knots are all isotopic and cannot be affine knots. So it follows from their proof that writhe number of affine real rational knots is strictly less than the bound for the projective case but it is not yet clear if and why it can be much smaller in affine real rational knots.
\end{rmrk}

\begin{remark}
In~\cite{mikhalkin}, it was proved that the maximally writhed projective real rational knots are all isotopic. However, in the case of affine glued knots, this is not true: there can be more than one isotopy class of glued affine knots with maximal writhe. Figure~\ref{same_writhe} shows two degree~6 knots of maximal writhe, i.e. 4, but which are not isotopic because one is the unknot and the other is the trefoil.
\end{remark}

\begin{figure}[h]
\centering
\includegraphics[width=0.7\textwidth]{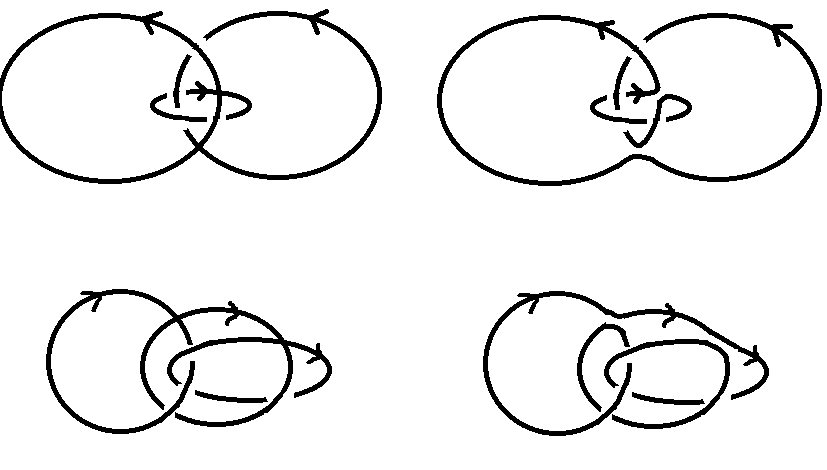}
\caption{Non-isotopic knots with the same maximal writhe}
 \label{same_writhe}
\end{figure}

\subsection{Linear skein relations and rigid pure links}
\label{skeinrigid}

Gluing reminds one of the skein relation because of a smoothing. Can we exploit the skein relation to understand skein relations of glued knots? %% please check this line
 Here, we show that we can relate the skein relation of glued knots with another special class of links.

The study of the isotopy of skew lines \cite{viroskew} of the union of non-intersecting lines in $\mathbb{RP}^3$ up to isotopy. The natural affine counterpart in $\mathbb{R}^3$ is the study of the isotopy of the unions of non-intersecting ellipses, which we will call rigid pure links, and since it is also a variety, it has a degree which turns out to be twice the number ellipses, and therefore:

\begin{figure}[h]
\centering
\includegraphics[width=0.5\textwidth]{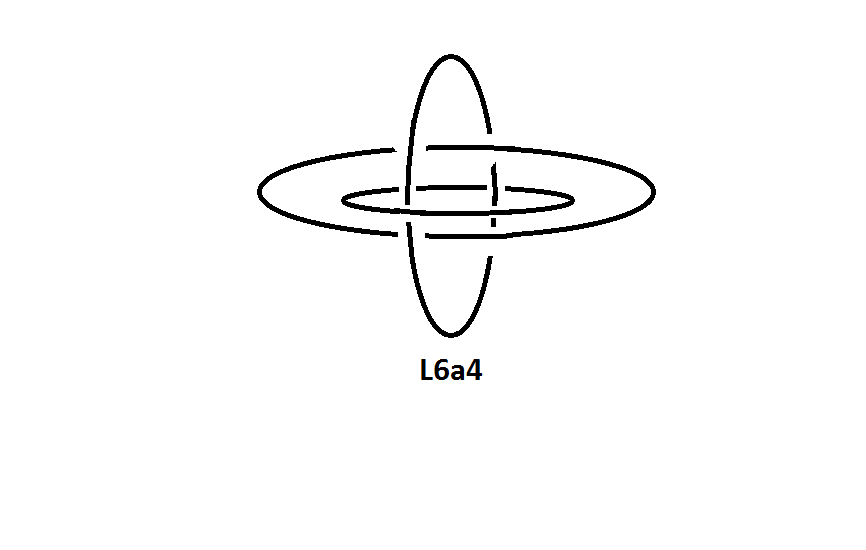}\hfill
\includegraphics[width=0.5\textwidth]{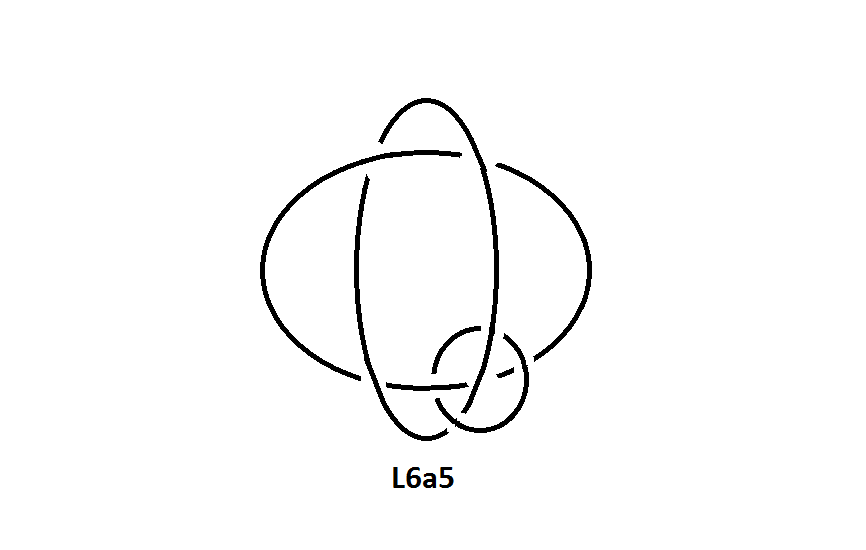}\hfill
\includegraphics[width=0.5\textwidth]{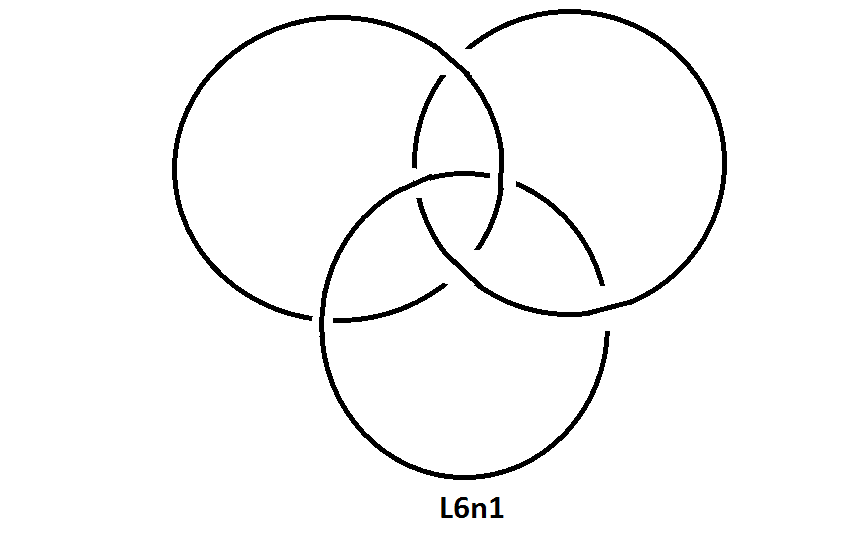}
%%\caption{Preglued ellipses corresponding to $3^{rd}$ and $6^{th}$ case.}
\caption{prime rigid links with three components.}
\label{rigidpure}
\end{figure}

\begin{definition}
  A union of $m$ non-intersecting ellipses in $\mathbb{R}^3$ is called a rigid pure link of degree~$2m$.
\end{definition}

``Rigid'' indicates that the components of the link are rigid ellipses and ``pure'' indicates that no component is knotted with itself, i.e. all crossings arise owing to multiple components. Not every link can be isotopic to a rigid pure link; the only two component rigid pure links can be the unlink and the Hopf link.

Rigid pure links can be subtle; for instance, it is not immediately clear from the usual diagram of the Borromean rings if it can be realized by rigid ellipses. However, the link diagram labelled L6a4 in Figure~\ref{rigidpure} is a way to realize the Borromean rings as a rigid pure link. Other rigid pure links with three ellipses are shown.

\begin{figure}[h]
\centering
\includegraphics[width=0.3\textwidth]{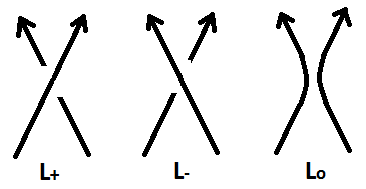}
\caption{Skein relations} 
\label{skein}
\end{figure}  

Consider a knot glued out of $m$ ellipses and, for each point of gluing, consider a small enough neighbourhood in which the knot looks like the $L_0$ shown in Figure~\ref{skein}. After a small perturbation, it would be homeomorphic to either $L_+$  or $L_-$. $L_+$ and $L_-$ are related by a crossing switch.

We can express a skein relation of a skein invariant, $P$, in the form,
\[aP(L_+) + bP(L_-) = P(L_0)\]

Consider an oriented preglued knot $G$ made of $m$ ellipses and fix a planar projection. Note that if we smoothen out the glued points in a manner compatible with the orientation we obtain a glued knot $K_G$, of degree~$2m$. However, if instead of smoothing out the intersection points, they were perturbed, we would get rigid pure links. There are two choices of perturbation for each glued point: one which will result in a crossing of local writhe $+1$ and the other of local writhe $-1$. Therefore, given an oriented preglued knot $G$, for each map $s : X \to \{-1, 1\}$, where $X$ denotes the set of glued points, there is a rigid pure link, which we will denote by $G_s$, obtained by perturbing a glued point $x$ so that a neighbourhood of $x$ is a crossing of sign $s(x)$. 

Given $s$, let $n_+(s)$ denote the number of glued points for which $s$ is positive, and $n_-(s)$ denote the number of glued points for which $s$ is negative. Therefore, $n_+(s) + n_-(s) = m-1$, i.e. the  number of glued points.  Then using the skein relation repeatedly, we obtain,

\[P(K_G) = \displaystyle\Sigma_s a^{n_+(s)} b^{n_-(s)} P(G_s)\]

Applying this to the case where $P=\nabla$, i.e. the Alexander-Conway polynomial, which satisfies the skein relation $1/z\nabla(L_+) - 1/z\nabla(L_-) = \nabla(L_0)$, the relation reduces to,

\[\nabla(K_G) = \displaystyle\Sigma_s 1/z^{n_+(s) + n_-(s)}\nabla(G_s) = (-1)^{n_-(s)} 1/z^{m-1} \displaystyle\Sigma_s \nabla(G_s)\]

We obtain the following restriction on the Alexander-Conway polynomials of glued knots of degree~$2m$ if we know the Alexander-Conway polynomials of all the rigid pure links made out of $m$ ellipses,

\begin{lema}
  If a knot is glued out of $m$ ellipses then the the degree of its Alexander-Conway polynomial is bounded above by the maximum degree of the Alexander-Conway polynomial of a rigid pure link divided by $z^{m-1}$. \qedhere
\end{lema}

On the other hand, if we know the Alexander-Conway polynomials of all the knots of degree~$2m$, then we can confirm the existence of a rigid pure link of at least a certain degree.

\begin{coro}
If there exists a knot that can be glued out of $m$ ellipses whose Alexander-Conway polynomial has degree~$d$, then there exists a rigid pure link whose Alexander-Conway polynomial has at least degree~$d$.
\end{coro}

As a trivial consequence of this, we will see in section~\ref{classification} that the only knot that can be glued out of two ellipses is the unknot, which has Alexander-Conway polynomial 1, and therefore there must exist a rigid pure link made out of two ellipses which has Alexander-Conway polynomial of degree at least 1. Indeed, the Hopf link has this property. We will also see that glued knots of degree~6 can be either the trefoil or the figure 8, both of which have Alexander-Conway polynomials of degree 2. Therefore, there must exist a rigid pure link made out of three ellipses whose Alexander-Conway polynomial is of degree at least 4.
%We can orient each ellipse of the preglued knot. A choice of orientation for each ellipse will be called an orientation of the preglued knot. Therefore, there are $2^m$ possible orientations of a preglued knot. The main reason for defining a preglued knot is that for each orientation of a preglued knot, we obtain a glued knot by gluing $P_i$ to $P_{i+1}$ so that their orientations are compatible. Note that if we reverse the orientation of each ellipse, we get the same glued knot, therefore there are at most $2^{m-1}$ glued knots associated to a given (unoriented) preglued knot.

\subsubsection{Kauffman bracket}
The Kauffman bracket formulation of the Jones polynomial defines a different skein relation that can also be exploited to relate the degree of the Jones polynomial of glued knots with the degree of the Jones polynomial of rigid links.

To express the Jones polynomial of a rigid link in terms of the Jones polynomial of a glued knot, we will need the following lemma:

\begin{lema}
Any rigid pure link is isotopic to a rigid link obtained as the perturbation of a preglued knot.
\end{lema}

\begin{proof}
  Consider a rigid $n$-component link $L_1$ made up of $n$ ellipses. For any $v\in \mathbb{R}^3$ and $t\in \mathbb{R}^{\geq 0}$, define the affine transformation $T_t^v(x) = x + tv$. Choose an ellipse, which we will denote $P_1$, from this link and choose a generic $v$ and $t_0$ that $T_{t_0}^v(P_1)$ intersects the rest of $L$ in exactly one point, but $T_{t}^v(P_1)$ does not intersect $L$ for any $t < t_0$. So $P_1$ intersects exactly one ellipse from $L$ in one point to form a preglued curve $P_2$ made out of two ellipses that links with the the $L_1$ minus these two ellipses, which we will call $L_2$. Observe that by construction, $L_1$ is a small perturbation of $P_2$ linked with $L_2$.

  Now apply the same procedure to $L_2$ that $T_{t_0}^v(P_2)$ intersects the rest of $L_2$ in exactly one point, but $T_{t}^v(P_2)$ does not intersect $L_2$ for any $t < t_0$.  This will form a preglued curve made out of 3 ellipses linking with a link, $L_3$, with $n-2$ components. Repeating this procedure we obtain a preglued knot $P_n$ with $n$ ellipses and $L_n$ is empty, so that $L_1$ is a  perturbation of $P_n$.
\end{proof}

Given a glued knot $K$, the lemma allows us to consider a pre-glued knot $G'_K$ so that the glued knot, $K$, is a perturbation of it. Let $X$ denote the set of ellipses of the pre-glued knot and let $\sigma : X \to \{-1,1\}$ denote a choice of orientations on each ellipse. The choice of orientation will determine whether the smoothing of each glued point is an $A$ smoothing or a $B$ smoothings. Let $G'_\sigma$ denote the knot obtained by a such a smoothing. We then get the following relations between the Kauffman brackets of the diagrams obtained by a fixing a generic projection:

\[\langle K \rangle = A^{n_+(\sigma) - n_-(\sigma)} \langle G'_\sigma \rangle\]

However, the Kauffman bracket is not an invariant because it depends on the diagram but it differs from the Jones polynomial by multiplication by $A^{3w(D)}$. We know that the writhe number of a glued knot is between $-(2{m\choose 2}-(m-1))$ and $2{m\choose 2}-(m-1)$ so the degree of the Jones polynomial of a diagram obtained by projecting a glued knot can differ from its Kauffman bracket by at most $2{m\choose 2}-(m-1)$. Note that $n_+(\sigma) - n_-(\sigma)$ is bounded above by $m-1$, i.e. the number of glued points. Therefore, if $d$ is the largest degree of the Jones polynomial of glued knots glued out of $m$ ellipses, then the degree of $\langle K \rangle$ is bounded above by $m - 1 + 2{m\choose 2}-(m-1) + d = 2{m\choose 2} + d$. Finally, the maximum writhe number of a rigid link is $2{m\choose 2}$, so after multiplying by 3 times the writhe number we get,

\begin{theorem}
If $d_m$ denotes the maximum degree of the Jones polynomial of glued knots glued out of $m$ ellipses, then the degree of the Jones polynomial of the rigid glued link is bounded above by $8{m\choose 2} + d_m$
\end{theorem}

%This time we apply the Kauffman bracket skein relation and obtain the Kauffman bracket of the rigid links in terms of the Kauffman brackets of the glued knots.  Consider a preglued knot described by the previous lemma. Repeatedly applying the skein relation at the points of pregluing, we express the Kauffman bracket of the rigid link in terms of the $2^{n-1}$ knots that can be glued out of the $n$ ellipses in the following manner.
%
%
%Using the above lemma, we express the rigid link as a perturbation of a preglued knot $K$. We can define a state, $s$, of the preglued knot by assigning a +1 or -1 to each point of gluing. Let $K_s$ denote the glued knot obtained by gluing choosing the gluing 
%
%
%    Choose an orientation of the first ellipse $P_1$ of the preglued knot. Now consider the point of intersection of $P_1$ and $P_2$. The given link is a perturbation of the preglued knot and this point of gluing perturbs to a crossing of a diagram. Choose the orientation so of $P_2$ so that on gluing, the smoothing is an $A$ smoothing if this point of gluing was assigned a +1 and a $B$ smoothing if the point of gluing was assigned a -1.
%
%This shows that the span of the Kauffman bracket, and therefore, of the Jones polynomial of a rigid link with 

\subsection{Rigid diagrams}
\label{rigid}
The diagrams obtained by projecting a glued knot are a natural subset of the possible diagrams that represent the isotopy class of that knot. Once can frame questions in terms of diagrams of knots in terms of these ``rigid'' diagrams. For instance, 

\begin{enumerate}
\item Is there a rigid diagram of a knot glued of minimal gluing degree with the number of crossings matching the crossing number of the knot?
\item Does every alternating knot have a rigid diagram which is alternating?
\item The maximum crossing number of a rigid diagram of a glued knot of degree~$2m$ is $c_{max}:=\frac{(2m-1)(2m-2)}{2}$. Is it possible to have a glued knot with minimal crossing number equal to $c_{max}$?
\end{enumerate}

\begin{example}
Figure~\ref{figureeight1} shows a rigid diagram of a figure eight knot which has the minimum number of crossings (4 crossings) and minimal gluing degree~6.  That its minimal gluing degree is 6 follows from the classification in section~\ref{classification}.
\end{example}

\begin{proposition}
  Let $c_m$ denote the number of crossings of a rigid diagram of a glued knot of degree~$2m$, then $c_m \cong m - 1 \textrm{ mod 2}$.
\end{proposition}

\begin{proof}
  If the knot is glued out of $m$ ellipses, then the number of glued points is equal to $m-1$. If we instead perturb rather than smoothed out the pre-glued knot, then these $m-1$ glued points along with the other crossings of the glued knot would form the crossings of the resulting perturbed rigid pure link. But the number of crossings of a rigid pure link is always even because each pair of ellipses must intersect in even number of points. Therefore, $m-1 + c_m$ is even, and the result follows.
\end{proof}

\begin{corollary}
The trefoil cannot have a minimal degree rigid diagram with minimal crossing number.
\end{corollary}
\begin{proof}
Since, as we will show in section~\ref{classification}, the minimal degree of a trefoil is 6, by the proposition, the rigid diagram must have even number of crossings so it cannot be the minimal crossing number, which for the trefoil is 3.
\end{proof}

%Question: Can we get a diagram of a glued knot with minimum number of ellipses needed to make that knot such that the number of crossings in that diagram is same as crossing number of that particular knot? So, what we are exactly asking is that if we always have to perform some Reidemeister Moves to get the diagram with the minuimum number of crossings when we use minimum number of ellipses to make that knot. The answer to this question is yes for some cases and for some it is not true. For example, the figure eight knot can be made with three ellipses glued at two places. Also, with the help of following diagram in Figure $\ref{figureeight1}$, we can show that we can have a diagram with 4 crossings after gluing these 3 ellipses. Whereas trefoil can also be made by gluing at least 3 ellipses but in this minimal degree no diagram can has 3 crossings as total number of crossings is the total sum of number of double points-the total sum of glued points, and total number of intersections in a projection consisting of ellipses is always even and as we are gluing 3 ellipses, total number of glued points i.e. $n-1$ is also even, therefore no projection after smoothening the glued points can yield 3 crossings.   
\begin{figure}[h]
\centering
\includegraphics[width=0.5\textwidth]{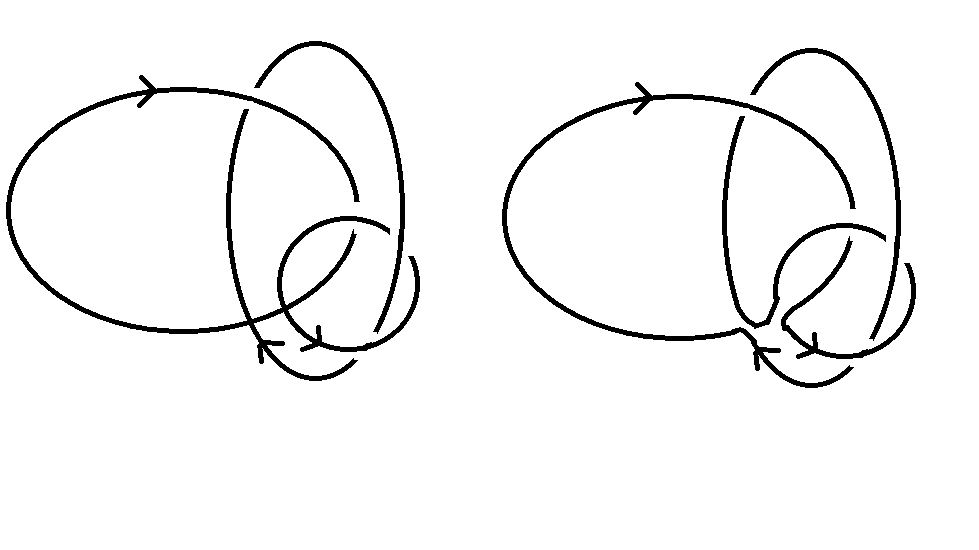}
\caption{Rigid diagram of a figure eight with the minimal number of crossings}
 \label{figureeight1}
\end{figure}

\blema Given an affine glued knot $K$ of degree 2n, where $n \geq 2$, with minimal crossing number $\frac{(2n-1)(2n-2)}{2}$ , we can always find a projection that will result in a non-alternating diagram.
\elema
\begin{proof}
The maximum crossings that a diagram of glued knot of degree $2n$ can have is $\frac{(2n-1)(2n-2)}{2}$. From the ellipses that are glued to form this knot, consider a pair of glued ellipses $E_m$ and $E_n$. By part (i) of lemma~\ref{ellipseintersect}, at least one of them is not being intersected by the other from the interior. Without loss of generality, assume that $E_m$'s interior is not intersected by $E_n$.  Project the knot from a point on the plane of $E_m$ so that $E_m$ is projected to a line segment and is therefore not a valid diagram.  By a sufficiently small perturbation, the projection will be a valid diagram  and $E_m$ will be projected  as a very thin ellipse which (before gluing) divides the plane into two components, such that the bounded component has no crossings. To obtain the maximum $\frac{(2n-1)(2n-2)}{2}$ number of crossings, $E_n $ would have to make 3 crossings with it (since one would be glued). in this projection, no strand of $E_n$ can make an alternating pattern of over-under with this thin ellipse because that would require it to link with $E_m$, however since $E_n$ does not intersect the interior of $E_m$, the linking number is 0. Therefore, this particular projection is non alternating. 
\end{proof}
\bcoro
There can be no prime alternating glued knot of degree $2n$, where $n\geq2$, that has minimal crossing number $\frac{(2n-1)(2n-2)}{2}$.
\end{coro}
\begin{proof}
As for each prime link $L$, the span of its Jones Polynomial is strictly less than the crossings of any non-alternating diagram $L'$ (cf. \cite{murasugi}), so if there exists a glued knot which is prime and has minimal crossing number $\frac{(2n-1)(2n-2)}{2}$, by the previous lemma, its span of the Jones Polynomial is strictly less than $\frac{(2n-1)(2n-2)}{2}$ and therefore there can not exist a prime alternating knot with the minimal crossing number.  
\end{proof}

At least for knots up to degree~6, for which the classification is worked out in section~\ref{classification}, the minimal crossing numbers are all much smaller than $C_{max}$. We do not yet know how much this bound can be reduced for general glued knots. On the other hand, for $m \geq 3$, we can prove the existence of glued knots of degree~$2m$ with minimal crossing number $2m-2$:

\begin{figure}[h]
\centering
\includegraphics[width=0.5\textwidth]{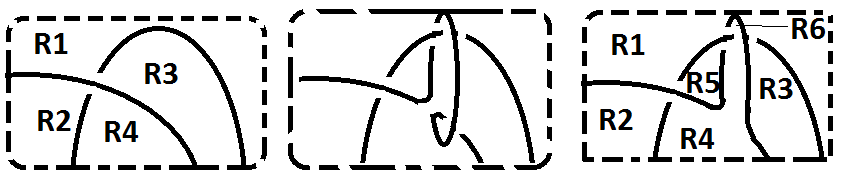}
\caption{Reduced diagrams}
 \label{compare}
\end{figure}

\begin{figure}[h]
\centering
\includegraphics[width=0.5\textwidth]{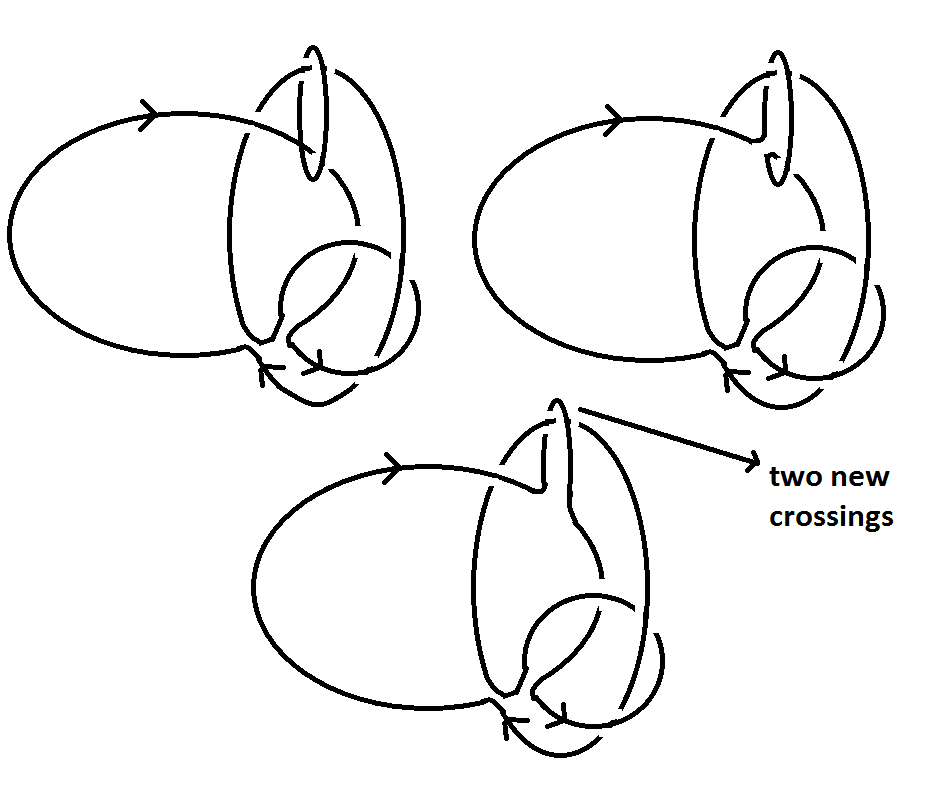}
\caption{Gluing one more ellipse to increase minimal crossing number by 2}
 \label{increase}
\end{figure}

\begin{theorem}
There exists a glued knot of degree $2n$ with crossing number $2n-2$ for $n\geq3$.
\end{theorem}
\begin{proof}
We will prove this by showing that for each  $n\geq 3$, there exists a glued knot of degree $2n$ which has a reduced alternating diagram with $2n-2$ number of crossings. As crossings in reduced alternating diagrams correspond to the crossing number of that particular knot, this will imply the result. For $n=3$,  we have figure eight knot as shown in Figure~$\ref{figureeight1}$. Note that the diagram shown is reduced alternating one. Now, suppose that there exist a knot $K_{2n}$ with degree $2n$ and it has a reduced alternating diagram with $2n-2$ crossings. Pick any one of its crossings and glue a new ellipse so that the new diagram changes only in a small neighbourhood around this crossing as shown in figure \ref{compare}. We could always choose the orientation of the ellipse such that it glues in this particular manner. Also, it is glued in such a way that the new diagram remains alternating after performing the First Reidemeister Move once. Moreover as the previous diagram was also reduced, this implies the regions surrounding each crossings in new diagram are all different and therefore the new diagram remains reduced. Figure \ref{increase} shows the procedure applied on figure eight knot.
\end{proof}
One can also observe in the construction of knots with the desired number of 3-colours that we get an increment of three crossings in the minimal crossing number by adding one more ellipse.

\section{Glued knots from ellipses}
%\begin{dfn}(defined previously)
 %We call a curve a preglued curve if it is the union of two oriented real rational knots that intersect in a single point and such that the tangents at those points are linearly independent.
%\end{dfn}
%\begin{dfn}
We get our glued knots or links after smoothing the intersection points in accordance with the orientations, therefore glued knots and links are determined up to oriented preglued curves.
%\end{dfn} 
Isotopic preglued curves made by gluing ellipses may give non isotopic glued knots. Two preglued curves on the left in Figure~\ref{isotopic_preglued}, are isotopic to each other, but after smoothing the gluing spot, they give unknot and trefoil (Figure \ref{smoothened}), which are not isotopic to each other. This is because the isotopies involving neighbourhoods of glued spots can be performed in preglued curves but are no longer valid in glued knots as after smoothing the topology in those small neighbourhoods of these two objects becomes different. This is why it is possible to 'switch crossings' in preglued curves as shown in the figure, but not in glued versions. Therefore, we will find all isotopic preglued curves up to non disturbed neighbourhoods around glued points to get all the glued knots. 
\begin{figure}[h]
\centering
\includegraphics[width=0.5\textwidth]{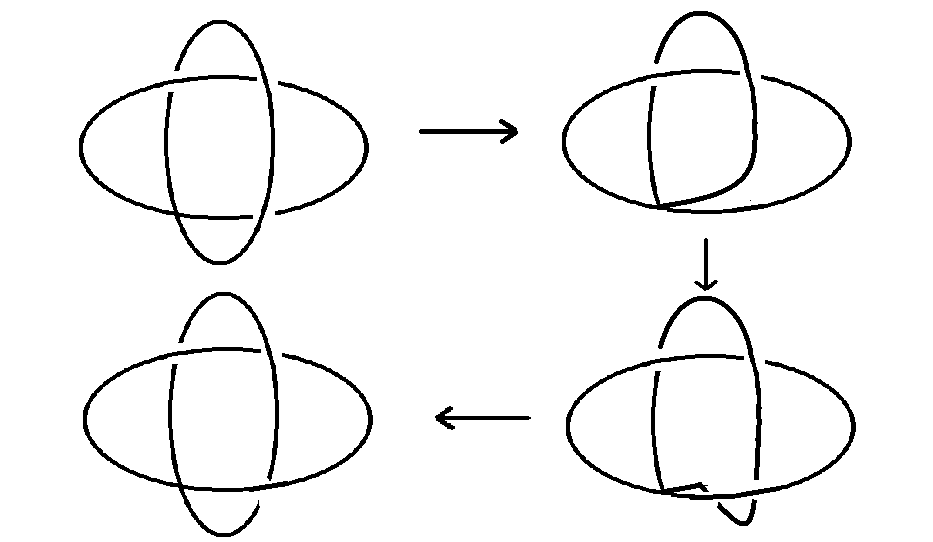}\hfill
\caption{Isotopic preglued curves}
 \label{isotopic_preglued}
\end{figure}

\begin{figure}[h]
\centering
\includegraphics[width=0.5\textwidth]{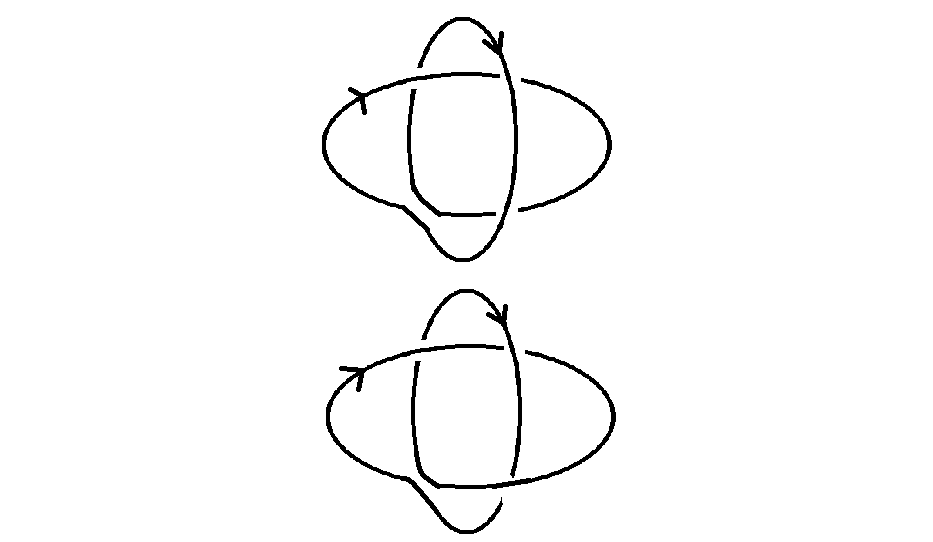}\hfill
\caption{Knots after smoothing preglued curves in Figure \ref{isotopic_preglued}.}
\label{smoothened}
\end{figure}

The gluing information can be presented in the form of a tree where each ellipse can be thought of as a vertex and two vertices share an edge if corresponding ellipses are glued. Trees for some corresponding preglued curves are shown below:%-\\\\\\
\begin{figure}[h]
\centering
\includegraphics[width=0.5\textwidth]{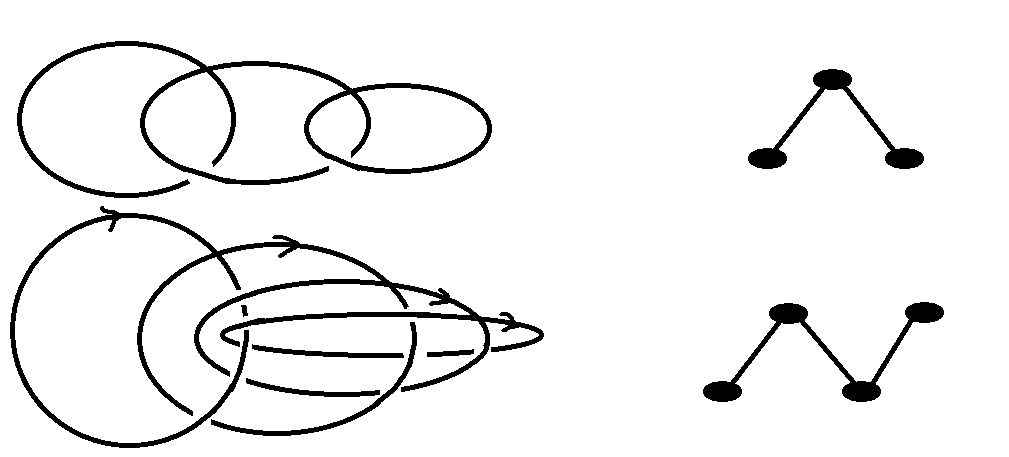}
\caption{Trees representing gluing information}
\end{figure}

%Classification of all the preglued curves with a particular number of ellipse can also be done by finding curves corresponding to these trees separately. There can be more than one non isotopic preglued curves (up to non disturbed glued spots) from the class of same tree. See Figure \ref{nonisotopic_sametree}. Its easy to check these curves give different glued knots but lie under same tree. It can be observed that the in one on them there is no linking between non-glued ellipses and in the other non-glued pair is linked. But embedding the linking information in the tree is also not enough to get unique elements. As the following two curves in Figure \ref{same_linking} have same tree and pair of non-glued ellipses are not linked but the knots they give after smoothing are figure eight knots and unknots respectively.\\
Classification of all the preglued curves with a particular number of ellipse can also be done by finding curves corresponding to these trees separately. There can be more than one non isotopic preglued curves(up to non disturbed glued spots) from the class of same tree. From Figure \ref{nonisotopic_sametree}, it is easy to check these preglued curves give different glued knots but lie under same tree. It can be observed that in one on them there is no linking between non-glued ellipses and in the other non-glued pair is linked. But embedding the linking information in the tree is also not enough to get unique elements. As the following two curves in Figure \ref{same_linking} have same tree and pair of non-glued ellipses are not linked but the knots they give after smoothing are figure eight knots and unknots respectively.

\begin{figure}[h]
\centering
\includegraphics[width=0.5\textwidth]{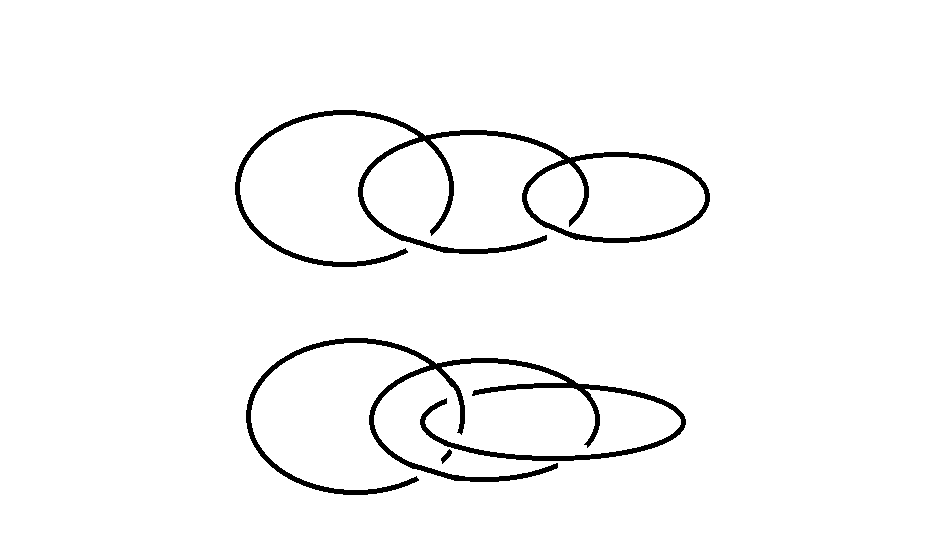}
\caption{Regularly non isotopic preglued curves}
\label{nonisotopic_sametree}
\end{figure}

\begin{figure}[h]
\centering
\includegraphics[width=0.7\textwidth]{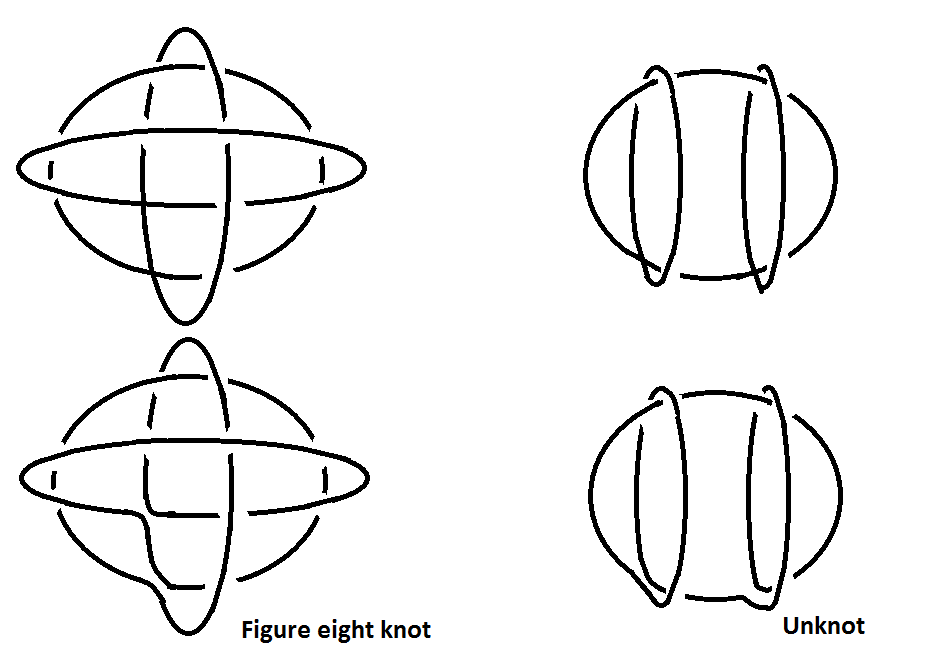}
\caption{regularly non isotopic preglued curves with same pairwise linking between ellipses}
 \label{same_linking}
\end{figure}

There is a difference between them too. In the curve which gives unknot, the pair of non glued ellipses are not passing through interiors of each other, whereas in the one which gives non trivial knot, although non-glued pair is not linked here also, but one of the ellipses from that pair passes through interior of other. The change is due to the fact that in first case there can be nothing which can provide a hindrance in pushing two ellipses apart from each other, in the other one, a strand can pass through, which can block a Reidemeister 2 move (see~Figure \ref{move}) on the diagram level. This is the same situation as in Borromean Rings. Unlinked components, if we see them pairwise, are forced to remain tangled by third component. Although this might not give complete criterion i.e. unique elements corresponding to a tree along with these information embedded during classification, this will help us in making cases when number of ellipses are low as we have done for the case of three ellipses in section $\ref{classification}$.

\begin{figure}[h]
\centering
\includegraphics[width=0.5\textwidth]{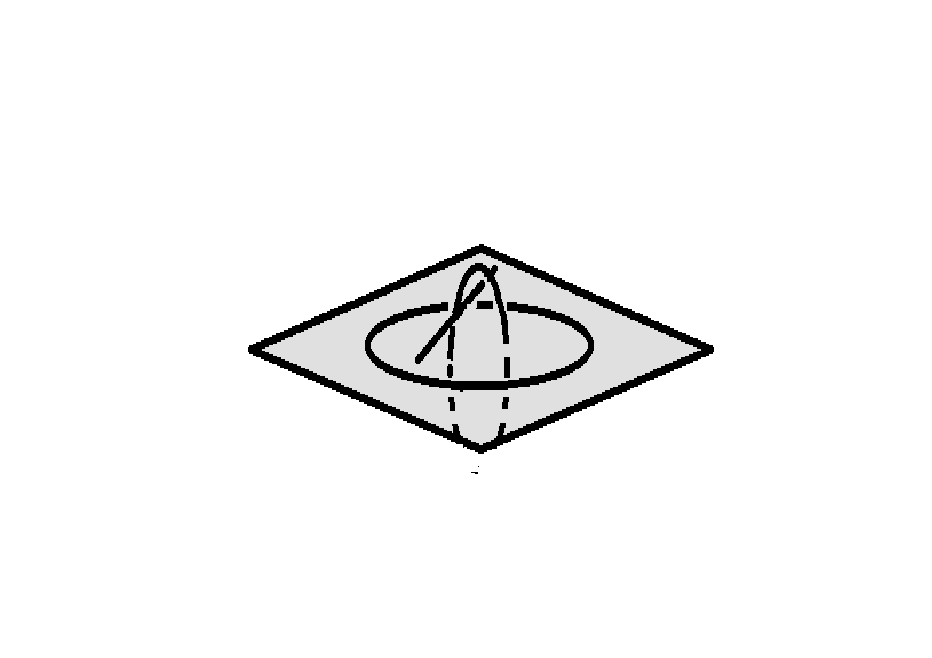}
\caption{blocked Reidemeister 2 move}
 \label{move}
\end{figure}
 %Now, two isotopic preglued curves will give us same glued knots if we take care that isotopy do not disturb some small neighbourhood around glued point. Perturbing the preglued curves from the intersection points will give us a set of $2^n$ different classical links unique upto isotopy, $n$ is the number of total intersections. We can define an invariant of preglued curves in terms of invariants of these classical links. The set of values of a particular invariant corresponding to these links will serve as an invariant of the preglued curve.

\section{Classification of glued Knots up to degree~6.}
\label{classification}
We show the existence of unknot, trefoil and figure eight in the following figure before showing that these are the only ones that we can make with number of ellipses less than or equal to 3.
\begin{figure}[h]
\centering
\includegraphics[width=0.5\textwidth]{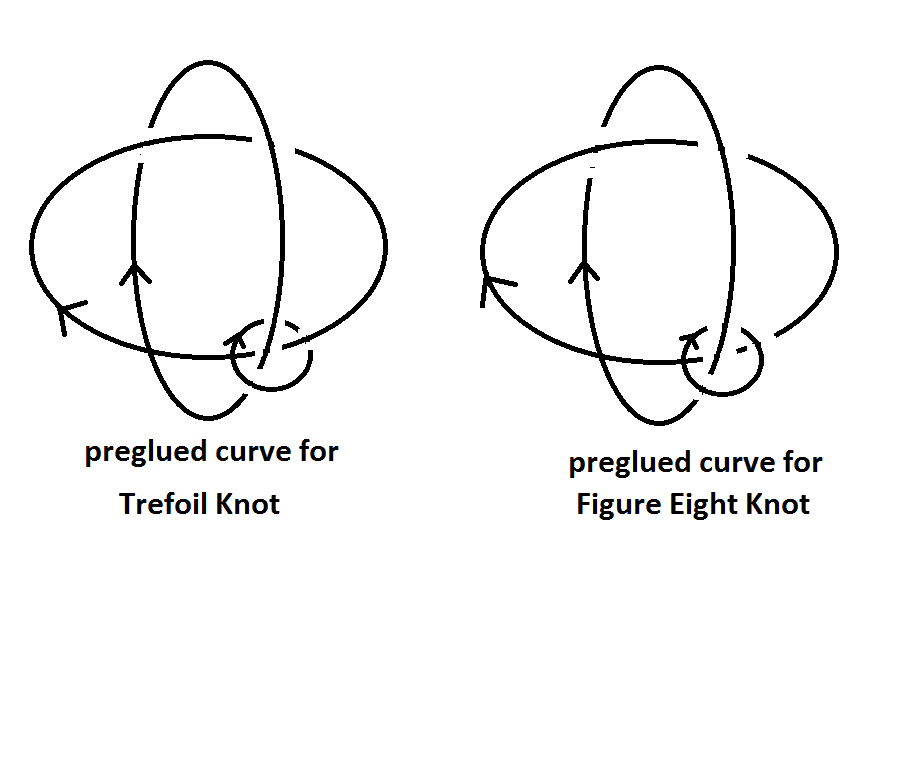}\hfill
\caption{Trefoil and Figure Eight}
 \label{trff8}
\end{figure}

\begin{lema}
Degree 2 glued curves are isotopic to the unknot.
\end{lema}

\begin{proof}
It is a single ellipse and therefore bounds a disc.
\end{proof}
 In fact, any two ellipses $E_1$ and $E_2$ are rigidly isotopic to each other. Translate the centre of $E_1$ to centre of $E_2$. Then rotate $E_1$ to align major and minor axis to that of $E_2$ and after that adjusting the length of these axes will take one ellipse to the other.

%Any ellipse can be transformed to some other ellipse by first rotating the plane of one ellipse to match it with the plane of other ellipse and then isotope it to other by expanding or shrinking it in plane. Now we will show that we can standardize any two glued ellipses into a particular form.

Any ellipse in space lies in a plane. It divides the plane into two components. We call the bounded component the \textbf{interior of the ellipse} and the unbounded component the \textbf{exterior of the ellipse}.
\begin{lema}
  \label{ellipseintersect}
\begin{enumerate}[(i)]
\item Let an ellipse $E_1$ is glued to another ellipse $E_2$. Either they do not intersect the interiors of each other or if $E_1$ intersects the interior of $E_2$, $E_2$ will not intersect the interior of $E_1$.%\\
\item Let $E_1$ and $E_2$ be two ellipses which are not glued to each other, either they do not intersect the interiors of each other or if $E_1$ intersects the interior of $E_2$ in two points, then $E_2$ does not intersect the interior of $E_1$. Also, if $E_1$ intersects the plane of $E_2$ in one interior point and one exterior point of $E_2$ , $E_2$ also intersects the plane of $E_1$ in one interior point and one exterior point of $E_1$.
\end{enumerate} 
\end{lema}
\begin{proof}
\begin{enumerate}[(i)]
\item If none of them pass through each other, there is nothing to prove. So, take the case where $E_1$ intersects the interior of $E_2$.
\begin{figure}[h]
\centering
\includegraphics[width=0.7\textwidth]{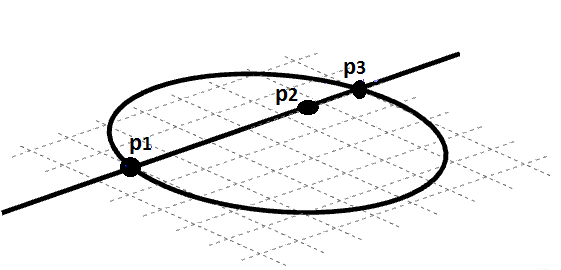}\hfill
\caption{$E_1$ intersecting plane of $E_2$ at $p_1$ and $p_2$. Figure is drawn using software GeoGebra }
 \label{ellipse_lemma(i)}
\end{figure}

Take the plane in which $E_2$ lies. $E_1$ will intersect with this plane in two points. One of them is the glued point and the other is somewhere in the interior of $E_2$. We will call these points as $p_1$ and $p_2$ respectively as shown in Figure \ref{ellipse_lemma(i)}. Draw a line passing through these two points such that it cuts $E_2$ in $p_3$. This is the common line of the planes of two ellipses and properly contains all the points that can possibly lie in the interior of $E_1$ and is also a point of $E_2$. Due to convexity of interior of $E_2$, $p_2$ lies between $p_1$ and $p_3$.  Now, $p_3$ can not be an interior point of $E_1$ because if it was, line segment from $p_1$ to $p_3$ should not contain any other point of $E_1$ due to convexity of interior of $E_1$. But, we know $p_2$ lies between $p_1$ and $p_3$. Therefore, $E_2$ is not intersecting interior of $E_1$.
\\\\\\
\item Suppose, without loss of generality, $E_1$ intersects the interior of $E_2$ in two points. We mark out these points of intersection as $p_1$ and $p_2$ and draw a line passing through them. See Figure \ref{ellipse_lemma(ii)}.\\
\begin{figure}[h]
\centering
\includegraphics[width=0.9\textwidth]{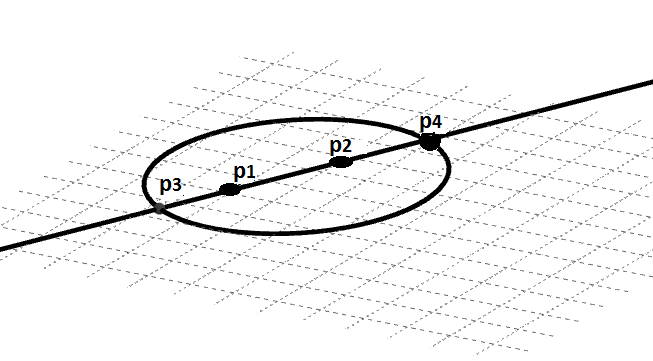}\hfill
\caption{$E_1$ intersecting interior of $E_2$ at $p_1$ and $p_2$. Figure is drawn using software GeoGebra.}
 \label{ellipse_lemma(ii)}
\end{figure}

This line will intersect $E_1$ in two points $p_3$ and $p_4$. and due to convexity of $E_1$, $p_1$ and $p_2$ lies between these $p_3$ and $p_4$. Let, $p_3$ be the one close to $p_1$ and $p_4$ as the one close to $p_2$.  Now,  if we see line segments joining $p_1$ to $p_4$, $p_2$  lies on it. Therefore, due to convexity of interior of $E_2$, $p_4$ can not lie in interior of $E_2$ as while joining a point on an ellipse with its interior point, no other point of that ellipse should lie in between. Similarly, $p_3$ can not be an interior point of $E_2$. Now, because this line has all common points shared by the planes of two ellipses, and also $p_3$ and $p_4$ are not interior points of $E_2$, interior of $E_2$ is not intersected by $E_1$.\\
In the last case, where $E_1$ intersects the plane of $E_2$ in two points, where one point say $p_1$ is interior point of $E_2$ and the other point say $p_2$ is the exterior point. We draw a line passing through these two points. It will cut the $E_2$ in two points say $p_3$ and $p_4$. Now, being an interior point of $E_2$, $p_1$ will lie between $p_3$ and $p_4$ and $p_2$, being an exterior point, will not lie between them. Therefore, if we join $p_1$ and $p_2$, one of $p_3$ and $p_4$ lie in between them which implies that point cuts the interior of $E_1$. Take it as $p_3$ without loss of generality.  Also, $p_1$ and $p_2$ will be on the same side of the component that we get after removing $p_4$. Therefore, $p_4$ can not be an interior point of $E_1$, as at least one of the line segments joining $p_1$ to $p_4$ or $p_2$ to $p_4$ will cut $p_2$ and $p_1$ respectively.
\end{enumerate}    
\end{proof}
\blema
%\vspace{0.1cm}

\begin{enumerate}[(i)]
\item Let $c$ be a chord of an ellipse $E$. If $c_1$ and $c_2$ be any two chords that are parallel to $c$ and on different sides of $c$, then $E$ is isotopic to the boundary of the region of the ellipse enclosed between $c_1$ and $c_2$.
\item Let $p$ be a point on an ellipse $E$. Then, $E$ can be isotoped from inside to a simple closed curve which lies completely inside an $\epsilon$- neighbourhood of the point $p$ and $ 0 < \epsilon < M$, where $M$ is the maximum distance of a point of $E$ from $p$.
\end{enumerate}
\elema

\begin{proof}
\begin{enumerate}[(i)]
\item The three parallel chords $c_1$, $c_2$ and $c$ are such that $c$ lies in between the other two as in Figure \ref{thinellipse}. It follows using Schoenflies Theorem that there exists an isotopy which is fixed outside the closure of region bounded by arc $\stackrel{\LARGE\frown}{MCN}$ and $c_1$ and takes arc $\stackrel{\LARGE\frown}{MCN}$ to $c_1$. Similarly, $\stackrel{\LARGE\frown}{PDQ}$ can be taken to $c_2$ to get the desired closed figure.\\
\begin{figure}[h]
\centering
\includegraphics[width=0.7\textwidth]{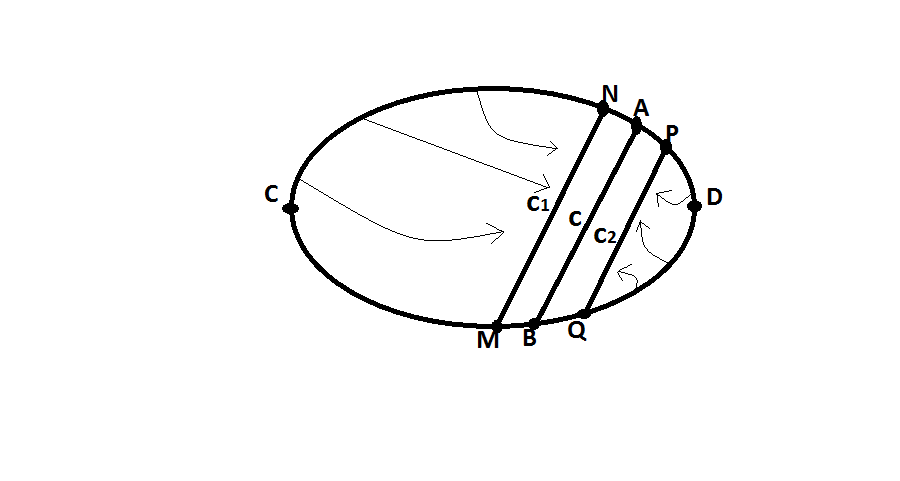}\hfill
\caption{Arcs of the ellipses isotoped to chords.}
 \label{thinellipse}
\end{figure}

\item One can adjust the width of the curve by using part (i) of the lemma. Then the length of the curve can be reduced so that it lies completely inside $\epsilon$ neighbourhood of $p$ for the given $\epsilon$ by moving the curve $ANMB$ as shown in Figure $\ref{scalingellipse1}$ towards $AB$.  
   % Let's see closely what we did during shrinking in the previous case. We do this by first making a chord passing through two points of the ellipse, one of which is the gluing point say $p$ as shown in Figure \ref{isotopy_lemma} . As we don't want to disturb the neighbourhood of the glued point, we take an $\epsilon$ width fringe around this chord, whose two opposites side are parts of ellipse and shrink the ellipse from both sides to the boundary of this fringe as shown in Figure \ref{isotopy_lemma}. This preserves the condition that the curve replacing previous ellipse can still intersect the plane of the other ellipse in at most two points. It should also be noted that if two ellipses are glued, shrinking process along a chord whose one end point is the glued point does not change their configuration i.e. if one is intersecting the other from interior, the curves will also do that at each step of shrinking. This is because one point of intersection of them with the other's plane is the glued point, and if the other point of intersection of one ellipse on the plane of the other ellipse is from inside, shrinking can't bring it outside and similarly if it intersects outside.\\
\end{enumerate}
\end{proof}
\begin{figure}[h]
\centering
\includegraphics[width=0.5\textwidth]{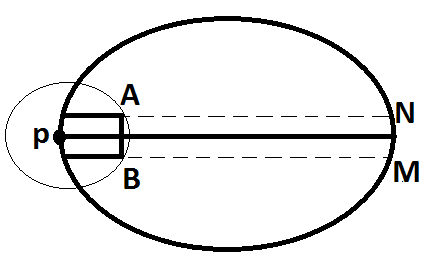}
\caption{Isotoping the ellipse in a small neighbourhood around one of its points. }
 \label{scalingellipse1}
\end{figure}

Now, for the case of two glued ellipses, by Lemma 4.2(i), interior of one of them will not be intersected by other and therefore by Lemma 4.3(ii) it can be isotoped to an ellipse that is in a small enough neighbourhood of the glued point so that in the projection it looks like a small kink in other ellipse after gluing. If we take a chord in this ellipse which does not intersect the other ellipse, infact we could istope it to a diagram with two crossings which implies we could get only unknot with two ellipses. 
\begin{figure}[h]
\centering
\includegraphics[width=0.7\textwidth]{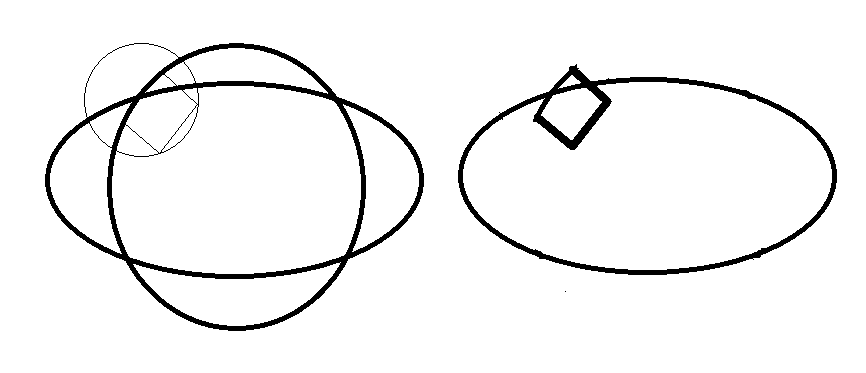}\hfill
\caption{We could get only unknot with two ellipses}
 \label{thinellipse}
\end{figure}
% how much should we write to show kink can be resolved within small neighbourhood.
%After gluing the ellipse which has been shrunk, will look like a miniscule kink in the bigger ellipse. So any two ellipses arranged in any manner will always give us unknots.
%\\\\\\
%We can show the same thing using only diagrams. With one ellipse, the only possiblity for crossing number is 0, with two ellipses crossing number is maximum 1, i.e. again only unknot is possible. This is because all the possiblities except $3^{rd}$ and $6^{th}$ as shown in Figure 26 will lead to unknot only and ellipses are not able to arrange in these two configurations because to have a diagram like these the second ellipse has to intersect plane of the first ellipse at least in three points.\\
 
Let us move to the case of 3 ellipses. %Using diagrams approach and  becomes a bit hard due to large number of cases have to be considered. We will try to use above observations to classify preglued curves upto regular isotopic.
 Here only 1 tree is possible. So, we will name the ellipses according to their position in tree, that is the ellipse which is in middle of tree (glued to both other ellipses) as middle ellipse or $E_2$ and the other two as left ellipse or $E_1$ and right ellipse or $E_3$. In other words, $E_1$ and $E_2$ are glued, $E_2$ and $E_3$ are glued, and, $E_1$ and $E_3$ are not glued to each other. Also, let $P_i$ be the plane in which $E_i$ lies. The following table show us the sub-parts in which we can divide our cases on the basis of how ellipses are interacting with interiors of others pairwise.

\begin{center}
\begin{tabular}{ | p{0.1\linewidth} | p{0.25\linewidth}| p{0.25\linewidth} | p{0.25\linewidth} |  } 
\hline
S.No.& Either the first or second ellipse intersects the interior of the other &  Either the second or third ellipse intersects the  interior of the other & Either the first or the third ellipse intersects the interior of other.\\
\hline
(i) & No & No & No \\
\hline
(ii) & No & Yes & No\\
\hline
(iii)& Yes & No & No\\
\hline
(iv) & Yes & Yes &  No\\
\hline
(v) & No & No & Yes\\
\hline
(vi) & Yes & No & Yes\\
\hline
(vii) & No & Yes & Yes\\
\hline
(viii) & Yes & Yes & Yes\\
\hline
\end{tabular}
\end{center}
Case (ii) is equivalent to Case (iii) and similarly (vi) is equivalent to Case (vii) as one can switch between them just by switching labels $E_1$ and $E_3$.%\\
%$\noindent{\bf Case (i), (ii) and (iii)\,}$

\case{ (i), (ii), and (iii)}
In the first three cases, $E_1$ and $E_3$ do not pass through interior of each other. Along with that, interior of at least one of them is not intersected by the middle ellipse also. So, the interior(s) of either $E_1$ (case 2) or $E_2$ (case 3) or both (case 1) is (are) completely non intersected by other strands. So, by Lemma 4.3(i), that particular ellipse can be isotoped via an isotopy that is supported only in the interior so that it lies in a small enough neighbourhood of the crossing arising from the glued point. Now, topologically that small loop can be resolved and now, after performing the same procedure as done with two glued ellipses, we get only unknots.

%$\noindent{\bf Case (iv)\,}$
\case{(iv)}
In this case, $E_1$ and $E_3$ do not pass through interiors of each other. Now, either of these types of sub-cases are possible:

\paragraph{Sub-case of Type 1} If $E_2$ has been intersected by strands of some other ellipses, by Lemma 4.3(i), the interiors of those ellipses are not intersected by $E_3$. So, again, we can isotope that ellipse from inside to get unknots eventually after gluing. Recall that we have not included a small neighbourhood during these contractions from inside, so all are isotopic.%\\ 

%\paragraph{Sub-case of Type 1} If $E_2$ has been intersected by strands of some other ellipses, by Lemma 4.3(i), the interiors of those ellipses are not intersected by $E_3$. So, again, we can isotope that ellipse from inside to get unknots eventually after gluing. Keep in mind, we haven't include a small neighbourhood during these contractions from inside, so all are isotopic.\\ 

\paragraph{Sub-case of Type 2} The only sub-part left is when both $E_1$ and $E_3$ are being intersected from the interior by $E_2$ simultaneously. Mark the glued points $g_1$ and $g_3$ indicating points of intersection of $E_1$, $E_2$ and $E_2$, $E_3$ respectively and two more points where $E_2$ intersects the interior of $E_1$ and $E_3$ as $p_1$ and $p_3$ respectively. Now, join $g_1$ and $p_1$ by a line segment. Similarly join $g_3$ with $p_3$. Extend these line segments to cut points say r and s of $E_2$ and $E_3$ respectively as shown in Figure \ref {case(iv)}. Name these extended line segments as $m$ and $n$. We will always get these r and s when we extend previous line segments because line segments which are extended are between points on the ellipse and interior points. Now, project ellipses along with these line segments. By using Lemma 4.3(ii), we replace $E_1$ and $E_3$ with closed curves whose projections lie in the small neighbourhood of projections these two line segments $m$ and $n$. These line segments can't intersect each other due to following result.
\begin{figure}[h]
\centering
\includegraphics[width=0.5\textwidth]{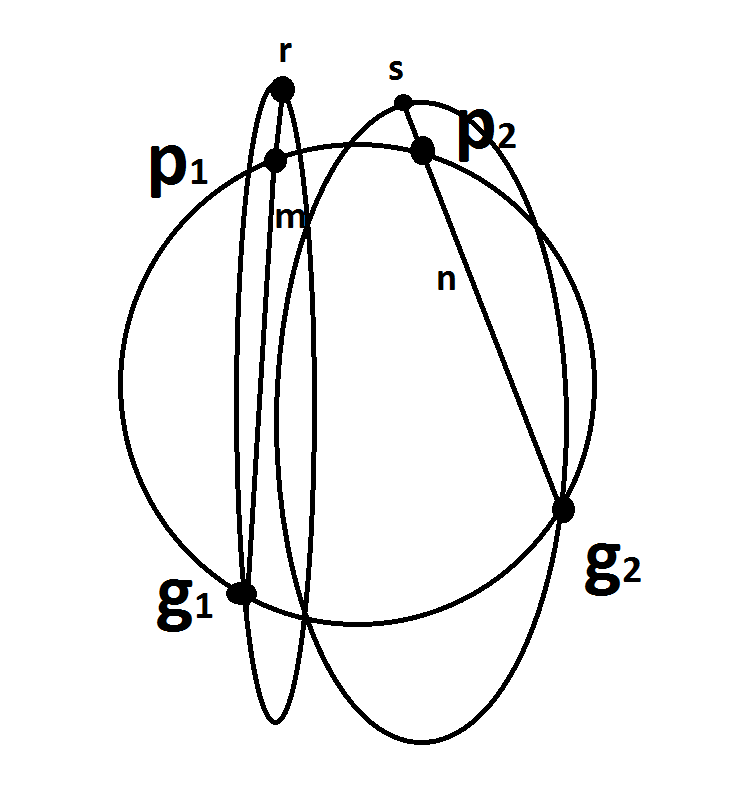}
\caption{Sub-case of Type 2 under Case (iv)}
 \label{case(iv)}
\end{figure}
%\vspace{10cm}

\blema
If we have two disjoint ellipses $E_1$ and $E_2$ in $\mathbb R^3$, such that none of them intersect the interior of the other, then their interiors are completely disjoint.
\elema
\begin{proof}
Take the ellipse $E_1$ and the plane $P_1$ in which it lies. Suppose some interior point  $p$ of it is also an interior point of $E_2$. It means a line segment $l$, with end points as points on $E_2$ pass through $p$ i.e. $p$ is an interior point of $l$. So, either these end points of $l$ lie on plane $P_1$ itself or one each in upper half and lower half of space which is divided by this plane. In any case, ellipse will cut the $P_1$ in at least two places. Mark them as $i_1$ and $i_2$. If any of these points lie inside the interior part of $E_1$, it means interior o ellipse $E_1$ is intersected by $E_2$. If they lie outside, then a line from $p$ to $i_1$ or $p$ to $i_2$ will comprise of interiors points of $E_2$ and is intersected by points of $E_1$. So, if some point lie in interior of both the ellipses, at least one of them would intersect the other in interior part.
\end{proof}
%\begin{rmrk}
%Same thing implies when two ellipses are glued. In that case, one of the $i_1$ or $i_2$ would be the glued point.  
%\end{rmrk}

So, projection of line segments $m$ and $n$ can not intersect as their preimages lie on same plane $P_2$ and do not intersect each other. Up to planar isotopy, the diagram with no crossing information will look like Figure \ref{case_(iv)}. As the two set of chords are parallel to $m$ and $n$, in the projection they can not give alternate over and under crossings. This is because for that chord has to cut the plane $P_2$. Therefore, those particular chords can be pulled apart either on the left or right hand side of $E_2$ to give only four crossings in the new diagram and out of those two will be disappeared after smoothing glued spots. As a result, we will always get unknots. %It is easy to see that we always get unknots because these parallel lines will have similar over- under crossing as with the case of two ellipses. 
\begin{figure}[h]
\centering
\includegraphics[width=0.5\textwidth]{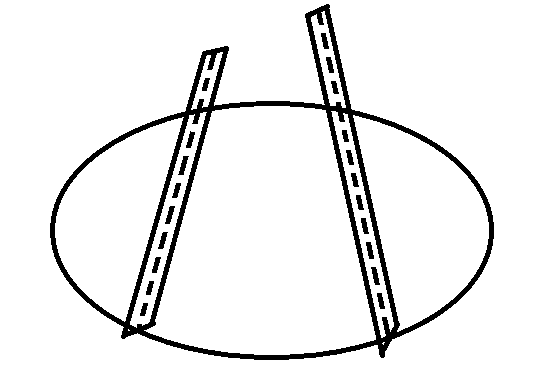}
\caption{Subcase of Type 2 under Case (iv) after isotoping ellipses $E_1$ and $E_3$}
 \label{case_(iv)}
\end{figure}

%Now, for the fifth case in the table, the middle ellipse neither intersect interior of other ellipses nor its interior is intersected by any other ellipse. Suppose we want to study what projections one can get in these cases. Actually, we will try to reduce each case to one of finite standard diagrams. We will try to shrink ellipses from inside as long as possible. We note that we can always shrink one of the three ellipses to almost a line. 
Before moving to case (v), we will observe some procedures that we will use very frequently in the next few cases. Note that the first two procedures can be followed on similar lines if $E_1$ and $E_3$ are interchanged.\newline

%$\noindent{\bf (Procedure 1)\,}$
\procedure{1}
($E_1$ and $E_3$ are linked and interior of $E_2$ is not intersected by any other ellipse.) \\Let the points where $E_3$ intersects $P_1$ be $p_1$ and $p_2$. Draw tangents to $E_3$ in the plane $P_3$ from these points. Name these as $T_1$ and $T_2$ and suppose they meet at a point $T$ ( in case they are parallel, we can ensure that they intersect by perturbing $E_3$ in the plane $P_3$). Take the projection of the preglued curve from $T$. The projection of $E_3$ will be a line segment with end points as one interior and one exterior point of projection of $E_1$. Now, shrink $E_2$ as in Lemma 4.3(ii) towards the chord $c$ joining two glued points of $E_2$. The projections(up to planar isotopy) will be of following types.
\begin{figure}[h!]
\centering
\includegraphics[width=1.0\textwidth]{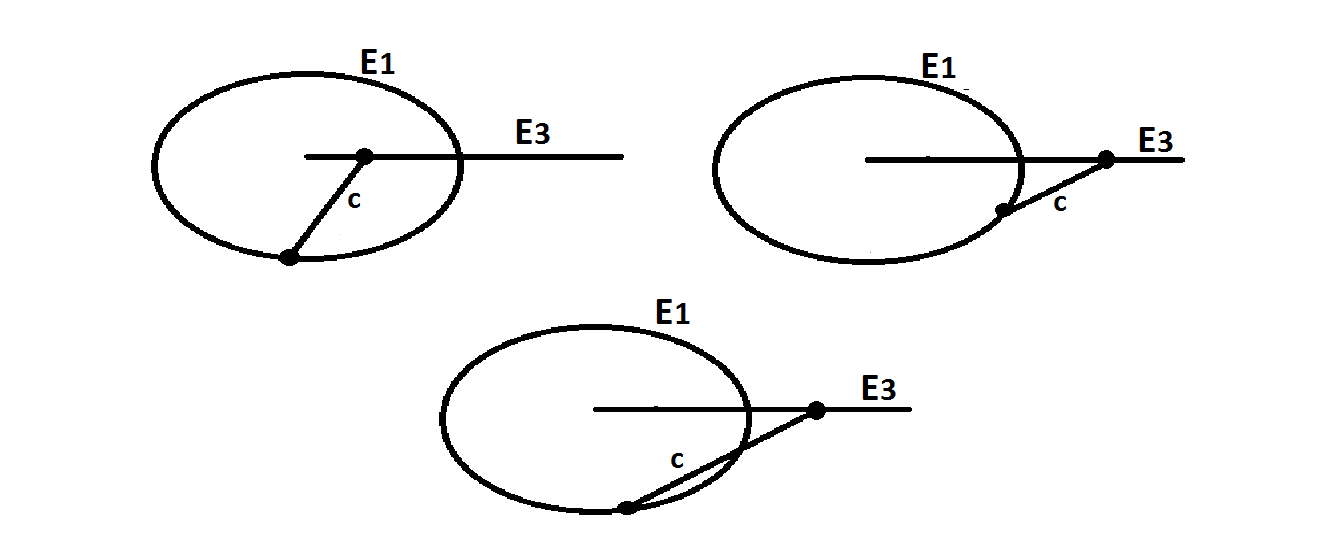}\hfill
\caption{Procedure 1.}
 \label{fig:procedure1}
\end{figure}

In the top two projections of Figure~\ref{fig:procedure1}, suitable perturbation of point of projection will reveal three ellipses with 6 double points and out of them 2 will be the glued points. Therefore, after smoothing the glued knot will have four crossings, which means that no knot with more than crossing number invariant four can be made by this case. In the third kind of diagram, by suitable perturbations of point of projections, eight double points will appear but we can always reduce it into a diagram with six double points. Note that with respect to the plane of shaded ellipse as shown in Figure \ref{fig:procedure1_3}, the chord of $E_2$ joining $E_1$ and $E_2$ will either lie completely towards the top-side or completely to the down side of the plane. This will determine the crossings near the point $p$ in the figure. If the chord is on the top side, we get the crossings near p as in $D_2$ and the one where chord is on the other side of the plane, we get crossings similar to $D_1$. Reidemeister moves can be performed in both diagrams to reduce the crossings up to 6 out of which two are glued spots. So, knots up to unknots, trefoils and figure eight knots are possible.       
\begin{figure}[h!]
\centering
\includegraphics[width=1.0\textwidth]{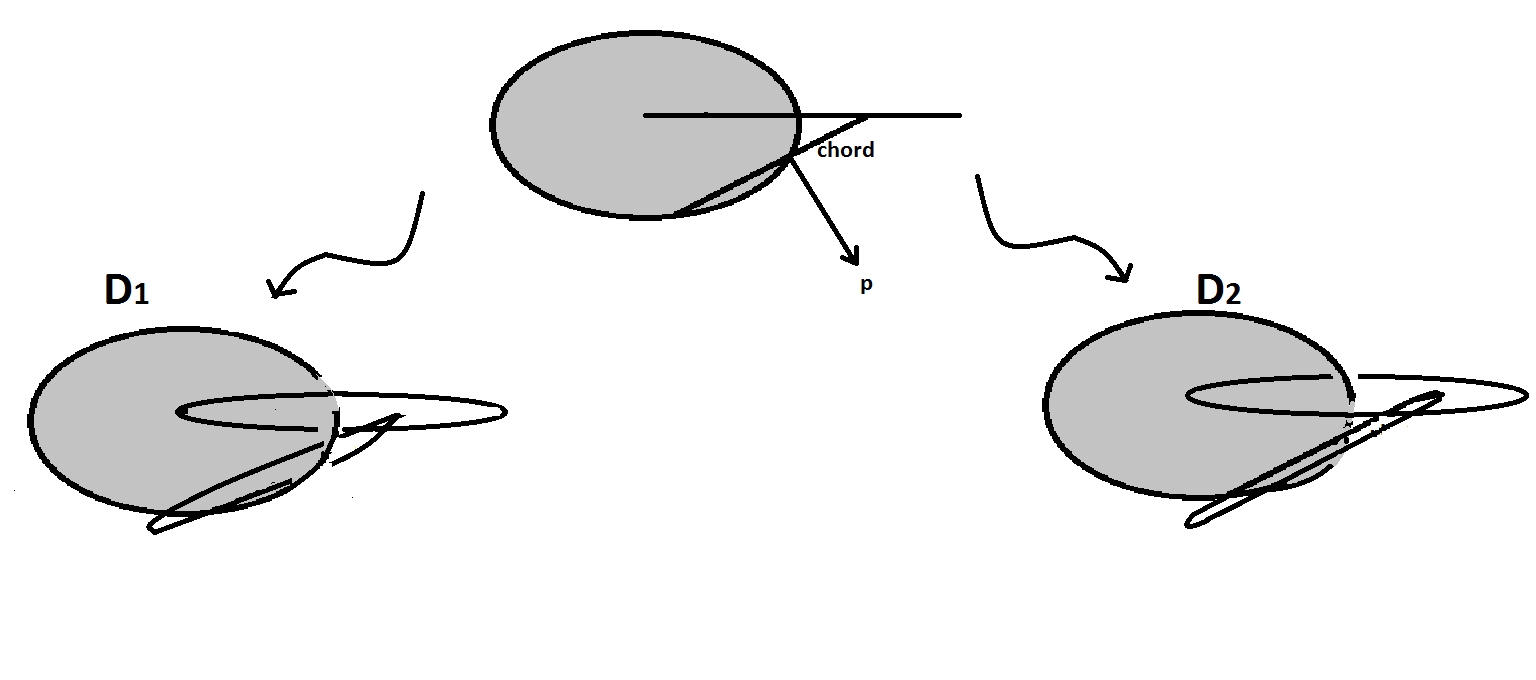}\hfill
\caption{Procedure 1(3).}
 \label{fig:procedure1_3}
\end{figure}

%$\noindent{\bf (Procedure 2)\,}$
\procedure{2}
($E_2$ intersects the interior of $E_1$ and interior of $E_3$ is intersected only by $E_1$ and in one point).\\
Let the points where $E_2$ intersects the plane of $E_1$ be $p_1$ and $p_2$. Draw tangents to these points on $E_2$ as done with $E_3$ in Procedure 1. Project from the point of intersections of these tangent lines. And finally with the help of Lemma 4.3(ii), it is possible to isotope $E_3$ to a curve which lies in the small neighbourhood of the chord $c$ which passes through the glued point of $E_3$ with $E_2$ and a point where $E_1$ intersects interior of $E_3$. Up to planar isotopy, the following projection will appear.
\begin{figure}[h]
\centering
\includegraphics[width=0.7\textwidth]{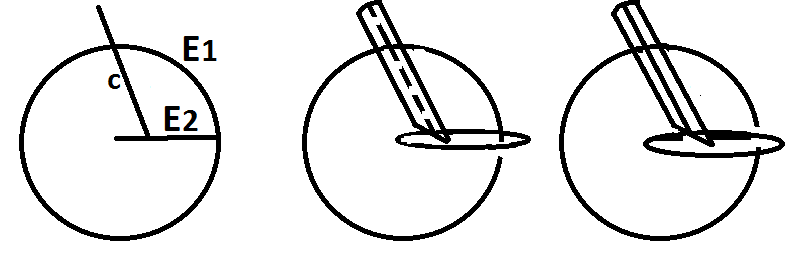}\hfill
\caption{Procedure 2.}
\end{figure}
Slightly perturbing the point of projections will reveal the 6 double points and after smoothing we will get 4 crossings. So, only unknot, trefoil and figure eight are possible.

%$\noindent{\bf (Procedure 3)\,}$
\procedure{3}
(Either one of $E_1$ or $E_3$ intersects the other's interior in two points and $E_2$ is not intersected by any other strand from inside).\\
Apply first step same as done in Procedure 1. The only difference in this case would be that instead of a line segment intersecting the projection of ellipse in the projection, we will get a line segment completely inside the projection of ellipse. 
\begin{figure}[h]
\centering
\includegraphics[width=0.7\textwidth]{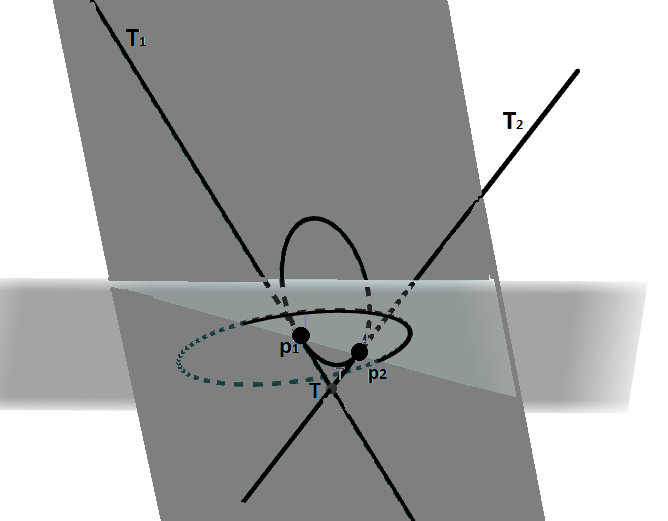}\hfill
\caption{ Initial step in Procedure 3. Figure is made using software GeoGebra.}
\end{figure} 
%\vspace{10cm}

Now, $E_2$ can isotoped, with the help of Lemma 4.3(i), towards a chord whose two end points are the two glued points with $E_1$ and $E_3$, and by an isotopy that is supported on the interior of the ellipse. So, the projection (up to planar isotopy) after perturbing the points of projection slightly will be as shown in Figure~\ref{procedure3}. In this case also maximum crossing number of knot that we can get after gluing is four.
\begin{figure}[h]
\centering
\includegraphics[width=0.7\textwidth]{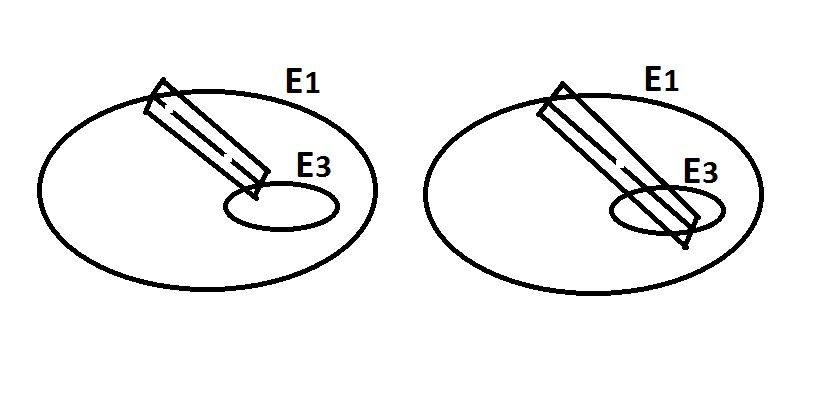}\hfill
\caption{Procedure 3.}
 \label{procedure3}
\end{figure}

%$\noindent{\bf Case (v)\,}$
\case{(v)}
In this case, pair of only $E_1$ and $E_3$ intersect from interiors and by Lemma 4.2(ii), they can do so in the following three ways (when they intersect):%\\

\begin{itemize}
\item $E_1$ intersects the interior of $E_3$ in two points.
\item $E_3$ intersects the interior of $E_1$ in two points.
\item Both $E_1$ and $E_3$ intersect each other's plane in one point from interior and one point from exterior. 
\end{itemize} 
$E_2$ is not intersecting the interiors of other ellipses. So, if one of $E_3$ or $E_1$ intersect the interior of other in two points, the interior of that particular ellipse is not being intersected by any other ellipse.  So, we can isotope it in a small neighbourhood around gluing point. For the third sub-case, Procedure 1 can be applied. This case also has potential to yield knots up to 4 crossings. 
\\

%\par
%$\noindent{\bf Case(vi)\,}$
\case{(vi)}
In this case, sub-cases can be further divided into two main types of sub-cases. The first in which $E_1$ intersect the interior of $E_2$ and second in which $E_2$ intersects the interior of $E_1$.

\paragraph{Sub-case of Type 1:} In the first type, the interior of $E_1$ can be intersected by $E_3$ only. If it happens that it is intersected by $E_3$ in two interior points, due to the fact that $E_2$ does not intersect $E_3$ and by Lemma 4.2(ii) the interior of $E_3$ is completely free (neither of $E_3$ or $E_2$ intersect each other from interior) to shrink it from inside and we will get an unknot. While on the other hand if $E_1$ intersects interior of $E_3$ in two points, then its interior is not intersected by other ellipses, therefore it is again an unknot after performing gluing procedure. If, let $E_1$ and $E_2$ are linked

So, the third ellipse has only one point in the interior where it is being intersected and the first ellipse intersects the middle ellipse's interior in one point. We will do the similar thing as done in Procedure 2, but instead of $E_2$ we will project $E_1$ as a line. Now, the projections (up to planar isotopy) can be of the types as in figure \ref{6_1_3}.
\begin{figure}[h]
\centering
\includegraphics[width=1.0\textwidth]{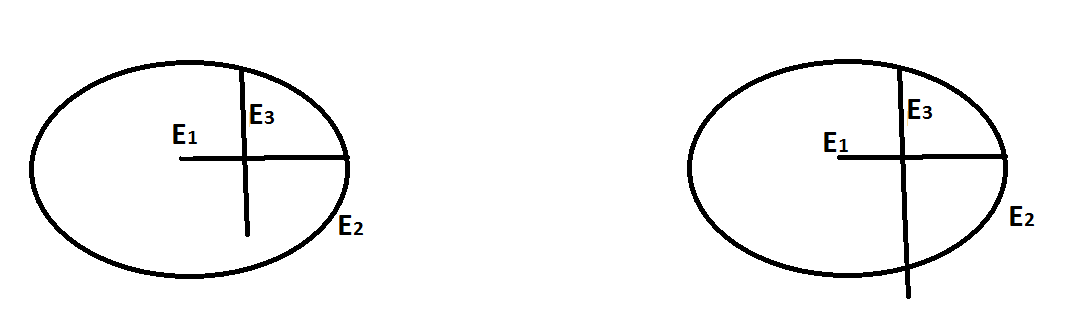}\hfill
\caption{Sub-case of Type 1 under case (vi)}
 \label{6_1_3}
\end{figure}

Its simple to inspect that with the given relations between three pairs of ellipses, this setup can easily be reduced to diagrams to knots with at most 4 crossings.

\paragraph{Sub-case of Type 2:} $E_2$ intersects $E_1$ from inside. In this case, $E_2$ is free from any intersection from its interior. Again, if $E_3$ is intersecting the first ellipse in two interior points, it can be shrunk towards glued point to get an unknot. If $E_1$ intersects interior of $E_3$ in two points, Procedure 3 can be applied. If, $E_1$ and $E_3$ are linked, Procedure 1 can be followed.%\\ %, say $p_1$ and $p_2$. We will make tangents $T_1$ and $T_2$ on first ellipse passing through $p_1$ and $p_2$. We can make sure $T_1$ and $T_2$ intersect at a point say $T$.%
\case{(vii)}
is the same as Case vi, we just have to replace $E_1$ with $E_3$ and vice versa.%\\

%\par
%$\noindent{\bf Case(viii)\,}$
\case{(viii)} Here all the three ellipses intersect interiors of each other. The table for sub-cases is as follows:-
\begin{center}
\begin{tabular}{ | p{0.1\linewidth} | p{0.25\linewidth}| p{0.25\linewidth} | p{0.25\linewidth} |  } 
\hline
S.No.& Relation between $E_1$ and $E_2$ intersecting from interior &  Relation between $E_2$ and $E_3$ intersecting from interior & Relation between $E_1$ and $E_3$ intersecting from interior.\\
\hline
1 & $E_1$ intersects $E_2$ & $E_3$ intersects $E_2$ & $E_1$ intersects $E_3$ in two points\\
\hline
2 & $E_1$ intersects $E_2$ & $E_3$ intersects $E_2$ & $E_3$ intersects $E_1$ in two points\\
\hline
3 & $E_1$ intersects $E_2$ & $E_3$ intersects $E_2$ & $E_1$ and $E_3$ are linked\\
\hline
4 & $E_1$ intersects $E_2$ & $E_2$ intersects $E_3$ & $E_1$ intersects $E_3$ in two points\\
\hline
5 & $E_1$ intersects $E_2$ & $E_2$ intersects $E_3$ & $E_3$ intersects $E_1$ in two points\\
\hline
6 & $E_1$ intersects $E_2$ & $E_2$ intersects $E_3$ & $E_1$ and $E_3$ are linked \\
\hline
7 & $E_2$ intersects $E_1$ & $E_3$ intersects $E_2$ & $E_1$ intersects $E_3$ in two points\\
\hline
8 & $E_2$ intersects $E_1$ & $E_3$ intersects $E_2$ & $E_3$ intersects $E_1$ in two points\\
\hline
9 & $E_2$ intersects $E_1$ & $E_3$ intersects $E_2$ & $E_1$ and $E_3$ are linked\\
\hline
10 & $E_2$ intersects $E_1$ & $E_2$ intersects $E_3$ & $E_1$ intersects $E_3$ in two points\\
\hline
11 & $E_2$ intersects $E_1$ & $E_2$ intersects $E_3$ & $E_3$ intersects $E_1$ in two points\\
\hline
12 & $E_2$ intersects $E_1$ & $E_2$ intersects $E_3$ & $E_1$ and $E_3$ are linked\\
\hline
\end{tabular}
\end{center}

Sub-cases 1 and 2 can be obtained from each other by replacing $E_1$ with $E_3$ and vice versa and so are 4 and 8, 5 and 7, 6 and 9, 10, and 11. Therefore, the sub-cases which are left to examine are 1, 3, 4, 5, 6, 10, and 12.

In the sub-cases 1 and 4, the interior of $E_1$ in not intersected by any other ellipse, and hence can be shrunk towards the glued points to get unknots after performing the gluing procedure. Sub-case 10 can be reduced to a diagram similar to that in Figure \ref{procedure3} by applying Procedure 3. In sub-case 3, we can standardize the projection as in a sub-case of type 2 of case 6 (Figure \ref{6_1_3}). In the $5^{th}$ sub-case, join the points of intersection of $E_2$ with the plane of $E_3$ and extend it to meet another point of $E_3$ to make a chord of $E_3$.  Project the setup from a point where tangents to $E_1$ passing through points where it intersects the plane of $E_2$. One of the points is inside $E_2$, and second is the glued point. Isotoping $E_3$ to get the projection along a small neighbourhood of the chord, up-to planar isotopy, will give the diagram as shown in Figure \ref{8_5subcase}.%\\\\\\\\\\\\\\\\\\\\\\\\\\\\

Again, all these can be reduced to diagrams with 4 crossings and therefore up-to figure eight can be achieved from this.

In the $6^{th}$ sub-case, Procedure 2 can be applied. $E_2$ intersects $E_3$ and $E_1$ can be isotoped along a chord with one of the points of $E_2$ (glued point) as an end point and passing through a point on $E_3$ (Figure \ref{subcase6of8}).% and  Project from the intersection points of tangents of $E_2$ passing through the points where $E_2$ intersects the plane of $E_3$. Join the glued point of $E_1$ and the point where it is being intersected by $E_3$ from interior and extend it to hit another point of $E_1$. Now, $E_1$ can be shrinked to a small neighbourhood around this chord. So, the setup will look like the following:-\\%

For the $12^{th}$ sub-case, Procedure 1 can be applied. (See Figure \ref{fig:procedure1}). 
\begin{figure}[h]
\centering
\includegraphics[width=0.5\textwidth]{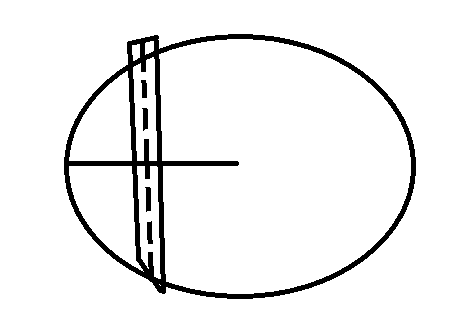}
\caption{sub-case 5 under case (viii)}
 \label{8_5subcase}
\end{figure}
\begin{figure}[h]
\centering
\includegraphics[width=0.5\textwidth]{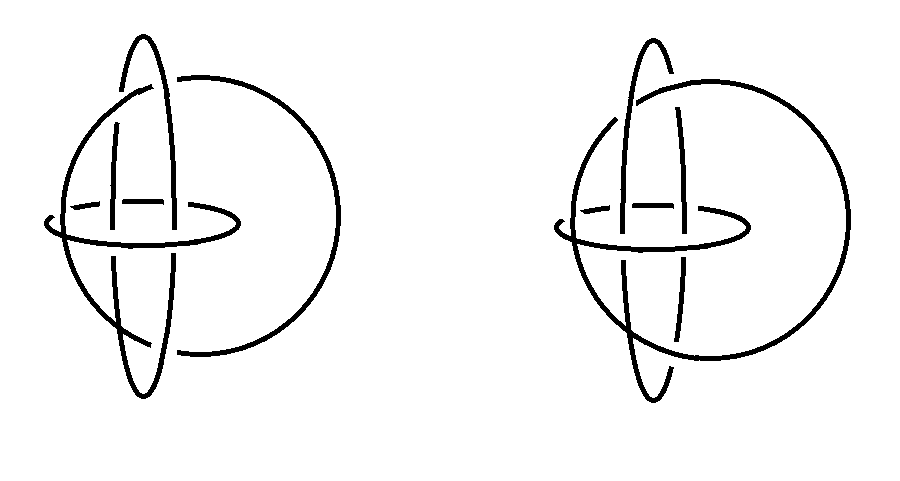}\hfill
\includegraphics[width=0.5\textwidth]{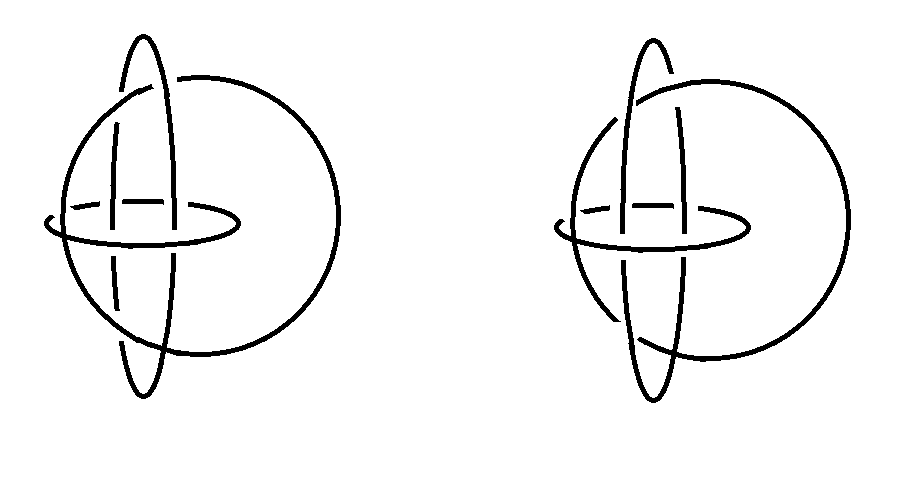}\hfill
\includegraphics[width=0.5\textwidth]{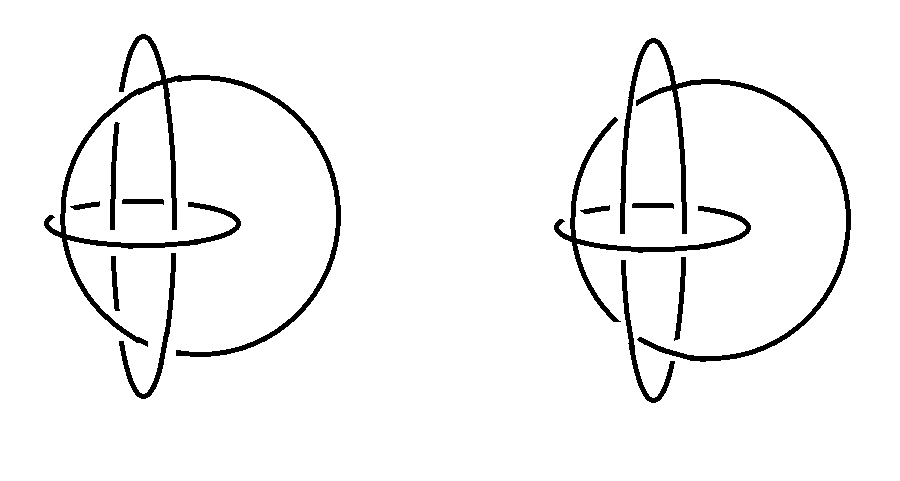}\hfill
\includegraphics[width=0.5\textwidth]{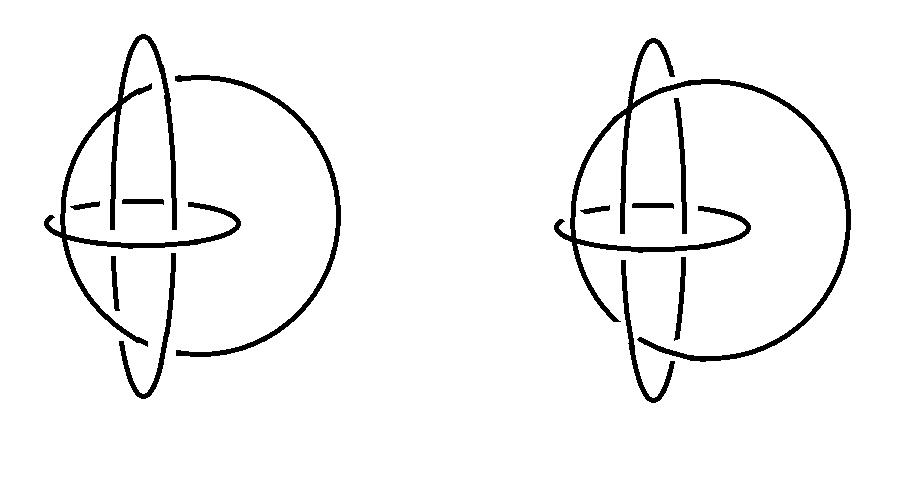}\hfill
\caption{some diagrams of preglued curves for sub-case 5 under case(viii)}
\end{figure}

\begin{figure}[t]
\centering
\includegraphics[width=0.5\textwidth]{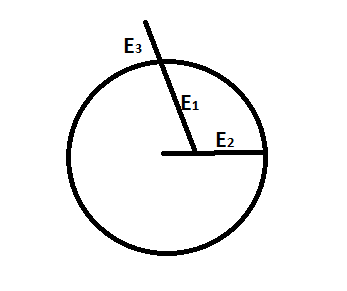}
\caption{Sub-case 6 of case 8}
 \label{subcase6of8}
\end{figure}
%\vspace{10cm}
  
\clearpage

\bibliographystyle{unsrt}
\bibliography{references}
\end{document}